\documentclass[a4paper,11pt]{article}
\usepackage{amsmath, amsfonts, amscd, amssymb, amsthm, enumerate, xypic}

\def\bfB{\mathbf{B}}

\DeclareMathOperator{\modu}{\operatorname{mod}}

\DeclareMathOperator{\Mat}{\operatorname{M}}
\DeclareMathOperator{\Hom}{\operatorname{Hom}}

\DeclareMathOperator{\Ker}{\operatorname{Ker}}
\DeclareMathOperator{\Mata}{\operatorname{A}}
\DeclareMathOperator{\Mats}{\operatorname{S}}
\DeclareMathOperator{\MatT}{\operatorname{T}}
\DeclareMathOperator{\MatD}{\operatorname{D}}
\DeclareMathOperator{\End}{\operatorname{End}}
\DeclareMathOperator{\Diag}{\operatorname{Diag}}

\DeclareMathOperator{\GL}{\operatorname{GL}}
\DeclareMathOperator{\Vect}{\operatorname{span}}

\DeclareMathOperator{\tr}{\operatorname{tr}}
\DeclareMathOperator{\maxrk}{\operatorname{maxrk}}

\DeclareMathOperator{\car}{\operatorname{char}}

\DeclareMathOperator{\rk}{\operatorname{rk}}

\renewcommand{\setminus}{\smallsetminus}
\renewcommand{\epsilon}{\varepsilon}

\def\K{\mathbb{K}}
\def\F{\mathbb{F}}
\def\R{\mathbb{R}}
\def\C{\mathbb{C}}

\newcommand{\D}{\mathbb{D}}

\def\calA{\mathcal{A}}
\def\calB{\mathcal{B}}
\def\calC{\mathcal{C}}
\def\calD{\mathcal{D}}

\def\calF{\mathcal{F}}

\def\calH{\mathcal{H}}

\def\calM{\mathcal{M}}

\def\calP{\mathcal{P}}

\def\calS{\mathcal{S}}
\def\calT{\mathcal{T}}

\def\calX{\mathcal{X}}
\def\calY{\mathcal{Y}}

\def\lcro{\mathopen{[\![}}
\def\rcro{\mathclose{]\!]}}

\theoremstyle{definition}
\newtheorem{Def}{Definition}[section]
\newtheorem{Not}[Def]{Notation}

\theoremstyle{plain}
\newtheorem{theo}{Theorem}[section]
\newtheorem{prop}[theo]{Proposition}
\newtheorem{cor}[theo]{Corollary}
\newtheorem{lemma}[theo]{Lemma}
\newtheorem{claim}{Claim}[section]

\theoremstyle{plain}

\theoremstyle{remark}
\newtheorem{Rems}{Remarks}[section]
\newtheorem{Rem}[Rems]{Remark}

\title{Spaces of triangularizable matrices}
\author{Cl\'ement de Seguins Pazzis\footnote{Universit\'e de Versailles Saint-Quentin-en-Yvelines, Laboratoire de Math\'ematiques
de Versailles, 45 avenue des Etats-Unis, 78035 Versailles cedex, France}
\footnote{e-mail address: dsp.prof@gmail.com}}

\begin{document}

\thispagestyle{plain}

\maketitle
\begin{abstract}
Let $\F$ be a field. We investigate the greatest possible dimension $t_n(\F)$
for a vector space of $n$-by-$n$ matrices with entries in $\F$ and in which every element is triangularizable over the ground field $\F$.
It is obvious that $t_n(\F) \geq \frac{n(n+1)}{2}$, and we prove that equality holds if and only if $\F$ is not quadratically closed or $n=1$, excluding
finite fields with characteristic $2$.
If $\F$ is infinite and not quadratically closed, we give an explicit description of the solutions with the critical dimension $t_n(\F)$, reducing the problem to the one of deciding for which
integers $k \in \lcro 2,n\rcro$ all $k$-by-$k$ symmetric matrices over $\F$ are triangularizable.
\end{abstract}



\section{Introduction}

\subsection{The problem}

Let $\F$ be a field. We denote by $\Mat_{n,p}(\F)$ the vector space of all $n$-by-$p$ matrices with entries in $\F$, and by
$\Mat_n(\F):=\Mat_{n,n}(\F)$ the algebra of all $n$-by-$n$ square matrices with entries in $\F$.
We denote by $\Mats_n(\F)$ its subspace of symmetric matrices, by $\Mata_n(\F)$ its subspace of alternating matrices
(i.e.\ skewsymmetric matrices with all diagonal entries zero), by $\mathfrak{sl}_n(\F)$
its subspace of matrices with trace zero, by $\MatT_n(\F)$ the space of all upper-triangular matrices of $\Mat_n(\F)$, and by
$\MatD_n(\F)$ the space of all diagonal matrices of $\Mat_n(\F)$.
Two subsets $\calX$ and $\calY$ of $\Mat_n(\F)$ are called \textbf{similar}, and we write $\calX \simeq \calY$, whenever there exists
$P \in \GL_n(\F)$ such that $\calY=P\calX P^{-1}=\{PMP^{-1} \mid M \in \calX\}$.

Let $n>0$ and $p>0$ be positive integers, and let $i \in \lcro 1,n\rcro$ and $j \in \lcro 1,p\rcro$.
We denote by $E_{i,j}$ the matrix unit of $\Mat_{n,p}(\F)$ all whose entries are zero with the exception of the one  at the $(i,j)$-spot, which equals $1$
(we understate $n$ and $p$ in this notation, as there will be no risk of confusion when we use it).

\begin{Def}
Let us say that a linear subspace of $\Mat_n(\F)$ is:
\begin{itemize}
\item  \textbf{weakly triangularizable} (respectively \textbf{weakly diagonalisable})
when all its elements are triangularizable over $\F$ (respectively diagonalisable over $\F$), i.e.\ similar to
a matrix of $\MatT_n(\F)$ (respectively, of $\MatD_n(\F)$);
\item  \textbf{strongly triangularizable} (respectively \textbf{strongly diagonalisable})
when it is similar to a subspace of  $\MatT_n(\F)$ (respectively, of $\MatD_n(\F)$).
\end{itemize}
We adopt similar definitions for subspaces of $\End(V)$ when $V$ is a finite-dimensional vector space over $\F$.
\end{Def}
Note that all those properties are preserved in replacing the said space of matrices by a similar space.

\begin{Not}
We denote by $t_n(\F)$ (respectively, by $d_n(\F)$) the greatest possible dimension for a weakly triangularizable (respectively, weakly diagonalisable)
linear subspace of $\Mat_n(\F)$.
\end{Not}

To start with, $\MatT_n(\F)$ is weakly triangularizable, so $t_n(\F) \geq \dbinom{n+1}{2}$.
As the intersection of an arbitrary weakly diagonalisable subspace with the space of all strictly upper-triangular matrices (none of which
is diagonalisable except the zero matrix) must be zero, we also have
$d_n(\F) \leq n^2-\dbinom{n}{2}$, to the effect that
$$d_n(\F) \leq \dbinom{n+1}{2}\leq t_n(\F).$$
Our first motivation for the problem of weakly triangularizable subspaces is that it is a natural variation of
Gerstenhaber's theorem on linear subspaces of nilpotent matrices \cite{Gerstenhaber,Serezhkin}, which
can be viewed as the problem of classifying ``weakly strictly triangularizable" subspaces.
Standard variations on Gerstenhaber's theorem generally deal with limitations on the spectrum of the matrices
\cite{OmladicSemrl,OmladicSivic,dSPlargerank,dSPfeweigenvalues}, or a lower bound for the multiplicity of the eigenvalue zero \cite{Atkinsonzeroeigenvalue}.

Another motivation resides in the observation of a collection of surprising differences between the situation of the real and complex fields.
Indeed, $\Mat_n(\C)$ is weakly triangularizable by the fundamental theorem of algebra, but is far from being strongly triangularizable.
Yet, by the Motzkin-Taussky theorem \cite{MoTau2} every weakly diagonalisable subspace of $\Mat_n(\C)$ is strongly diagonalisable.
Hence
$$d_n(\C)=n \quad \text{and} \quad t_n(\C)=n^2.$$
For real numbers, the situation is very different, as $\Mats_n(\R)$ is weakly diagonalisable with dimension $\dbinom{n+1}{2}$,
to the effect that $d_n(\R)=\dbinom{n+1}{2}$, and it is known \cite{Dobovisek,RandedSP,dSPSEVdiag2} that $\Mats_n(\R)$ is, up to conjugation, the
only weakly diagonalisable subspace of $\Mat_n(\R)$ with this dimension.
Moreover, since $\Mata_n(\R)$ contains no nonzero real-triangularizable matrix\footnote{Indeed every skewsymmetric real matrix is diagonalisable over $\C$
with eigenvalues in $i\R$.} and has dimension $\dbinom{n}{2}$, by intersecting we see that
$t_n(\R) \leq \dbinom{n+1}{2}$, and we conclude that
$$d_n(\R)=\dbinom{n+1}{2}=t_n(\R).$$
Hence, over the complex numbers it is loose to be weakly triangularizable, and very restrictive to be
weakly diagonalisable. But over the reals the greatest possible dimensions for those two properties are the same! And large weakly diagonalisable
subspaces of matrices exist over the reals, but not over the complex numbers.

In making these somewhat simple observations, we discovered that nothing was known on weakly triangularizable subspaces over general fields,
while there has already been substantial research on weakly diagonalisable subspaces \cite{RandedSP,dSPSEVdiag2}.
Therefore, we hope that this article will be a solid first step in this new direction of research.

\subsection{Main results}

Of course $t_n(\F)=n^2$ if and only if every polynomial in $\F[t]$ with degree $n$ splits over $\F$.
We are concerned here with the possibility that $t_n(\F)=\dbinom{n+1}{2}$, and we shall prove that it is in some sense the most common situation.
Before we state our result, it is critical that we clear out what is meant here by a quadratically closed field.

\begin{Def}
The field $\F$ is called \textbf{quadratically closed} when every polynomial of $\F[t]$ with degree $2$ splits over $\F$. \\
We say that $\F$ is \textbf{NRC} (for ``Non Root Closed") whenever $x \in \F \mapsto x^2 \in \F$ is nonsurjective.
\end{Def}

Alternatively, $\F$ is quadratically closed if and only if it has no algebraic extension of degree $2$.
When $\F$ has characteristic other than $2$, it is quadratically closed if and only if it is not NRC, but this is not true
over fields with characteristic $2$. For example, all finite fields with characteristic $2$ are perfect and hence non-NRC,
but none is quadratically closed. We can now state our first main result:

\begin{theo}\label{theo:dim}
Let $\F$ be an arbitrary field, and $n \geq 2$ be an integer.
Assume that $\F$ is not finite with characteristic $2$.
Then $t_n(\F)=\frac{n(n+1)}{2}$ if and only if $\F$ is not quadratically closed.
\end{theo}

Of course $t_1(\F)=1$ whatever the choice of $\F$.
Let us immediately explain the ``only if" implication. If $\F$ is quadratically closed then
every $2$-by-$2$ matrix with entries in $\F$ is triangularizable over $\F$, and hence every matrix of the form
$$\begin{bmatrix}
A & ? \\
0 & U
\end{bmatrix} \quad \text{with $A \in \Mat_2(\F)$ and $U \in \MatT_{n-2}(\F)$}$$
is weakly triangularizable.
The space of all such matrices has dimension $\dbinom{n+1}{2}+1$. And if $n$ is large we can of course create weakly triangularizable spaces whose
dimension largely exceeds $\dbinom{n+1}{2}$, by plugging additional $2$-by-$2$ cells of the form $\Mat_2(\F)$ on the diagonal.

For the remainder of the introduction, we assume that $\F$ is neither quadratically closed nor finite with characteristic $2$, so that $t_n(\F)=\dbinom{n+1}{2}$
(of course in the rest of the introduction we admit the validity of Theorem \ref{theo:dim}).

\begin{Def}
A weakly triangularizable subspace of $\Mat_n(\F)$ is called \textbf{optimal} when its dimension is $t_n(\F)$.
We adopt a similar definition for a weakly triangularizable subspace of $\End(V)$ when $V$ is an $\F$-vector space of dimension $n$.
\end{Def}

\begin{Def}
A linear subspace $\calS$ of $\End(V)$ is called \textbf{irreducible} when no nontrivial subspace of $V$ is invariant under all the elements of $\calS$.
\end{Def}

This prompts us to introduce the notation of the \emph{joint} of two spaces of square matrices.
For linear subspaces $\calA$ and $\calB$, respectively of $\Mat_n(\F)$ and $\Mat_p(\F)$ (with $n>0$ and $p>0$), we denote by
$\calA \vee \calB$ the set of all matrices of the form
$$\begin{bmatrix}
A & C \\
0 & B
\end{bmatrix} \quad \text{with $A \in \calA$, $B \in \calB$ and $C \in \Mat_{n,p}(\F)$,}$$
which we call the \textbf{joint} of $\calA$ and $\calB$.
Obviously, $\calA \vee \calB$ is a linear subspace of $\Mat_{n+p}(\F)$, it is weakly triangularizable if and only if so are $\calA$ and $\calB$,
and if $\dim \calA=\dbinom{n+1}{2}$ and $\dim \calB=\dbinom{p+1}{2}$ then $\dim (\calA \vee \calB)=\dbinom{n+p+1}{2}$.
Hence, under the assumption that $\F$ is neither quadratically closed nor finite with characteristic $2$,
the space $\calA \vee \calB$ is optimal if and only if $\calA$ and $\calB$ are optimal.
A general theory of optimal spaces actually yields the following result
(see Theorem \ref{theo:generalnonsense} in the appendix):

\begin{theo}\label{theo:generaldecomposition}
Let $\F$ be a field.
Let $\calS$ be an optimal weakly triangularizable subspace of $\Mat_n(\F)$.
Then there is a partition $n=n_1+\cdots+n_p$ into positive integers, and for every $i \in \lcro 1,p\rcro$
there is an irreducible optimal weakly triangularizable subspace $\calS_i$ of $\Mat_{n_i}(\F)$ such that
$$\calS \simeq \calS_1 \vee \cdots \vee \calS_p.$$
Moreover, the partition $(n_1,\dots,n_p)$ is uniquely determined by the similarity class of $\calS$,
and so is the similarity class of $\calS_i$ for each $i \in \lcro 1,p\rcro$.

Conversely, let $n=n_1+\cdots+n_p$ be a partition into positive integers, and for each
$i \in \lcro 1,p\rcro$ let $\calS_i$ be an irreducible optimal weakly triangularizable subspace of $\Mat_n(\F)$.
If in addition $t_n(\F)=\sum_{k=1}^p t_{n_k}(\F)+\sum_{1 \leq i<j \leq p} n_i n_j$,
then $\calS_1 \vee \cdots \vee \calS_p$ is an optimal weakly triangularizable subspace of $\Mat_n(\F)$.
\end{theo}

Hence Theorem \ref{theo:generaldecomposition} essentially reduces the classification of the optimal weakly triangularizable subspaces of matrices
to the classification of the \emph{irreducible} ones. Moreover, the condition $t_n(\F)=\sum_{k=1}^p t_{n_k}(\F)+\sum_{1 \leq i<j \leq p} n_i n_j$
is clearly satisfied if we have $\forall k \geq 1, \; t_k(\F)=\dbinom{k+1}{2}$, which will be the case here.

Now we can move on to our second main result.

\begin{theo}\label{theo:optimalcarnot2}
Let $\F$ be an infinite NRC field, and let $n \geq 1$.
Let $\calM$ be an irreducible optimal weakly triangularizable subspace of $\Mat_n(\F)$.
Then $\calM$ is similar to $\Mats_n(\F)$ (if $n=1$ then simply $\calM=\F$).
\end{theo}

If $\F$ is NRC and infinite, we know from Theorem \ref{theo:dim} that $t_n(\F)=\dim \Mats_n(\F)$.
And it is clear that $\Mats_n(\F)$ is irreducible (as classically $\Mats_n(\F) X=\F^n$ for all $X \in \F^n \setminus \{0\}$).
However $\Mats_n(\F)$ need not be weakly triangularizable.
Remember that the field $\F$ is called \textbf{Pythagorean} if the sum of two squares in $\F$ is a square in $\F$.
If $\F$ is not Pythagorean then there exist $a,b$ in $\F$ such that $a^2+b^2$ is not a square, and hence
$\begin{bmatrix}
a & b \\
b & -a
\end{bmatrix}$ has no eigenvalue in $\F$. Conversely, if $\F$ is Pythagorean and $\car(\F) \neq 2$ then it is clear from the same example
(because we can combine it linearly with $I_2$) that all the matrices of $\Mats_2(\F)$ are triangularizable.

Yet, if $\F$ is Pythagorean with $\car(\F) \neq 2$ but $\F$ is not formally real, then $\F$ is not NRC.
Indeed, in that case $-1$ is a sum of squares; then
every element $x\in \F$ can we written $x=\left(\frac{1+x}{2}\right)^2+(-1)\left(\frac{1-x}{2}\right)^2$
and hence is a sum of squares, and hence a square because $\F$ is Pythagorean.

Finally, if $\Mats_n(\F)$ is weakly triangularizable and $n \geq 2$, then $\Mats_{n-1}(\F)$ is weakly triangularizable (for
$M \in \Mats_{n-1}(\F)$, simply consider the extended matrix $M \oplus 0 \in \Mats_n(\F)$).

Hence, if $\F$ is infinite and NRC:
\begin{itemize}
\item Either every symmetric matrix over $\F$ is triangularizable, in which case for all $n \geq 1$
there is up to similarity exactly one irreducible optimal weakly triangularizable subspace of $\Mat_n(\F)$, namely $\Mats_n(\F)$;
\item Or there is a greatest integer $N \geq 1$ such that every symmetric $N$-by-$N$ matrix over $\F$ is triangularizable,
and then, for all $n \in \lcro 1,N\rcro$, there is up to similarity exactly one irreducible optimal weakly triangularizable subspace of $\Mat_n(\F)$, namely $\Mats_n(\F)$,
and for all $n>N$ there is no irreducible optimal weakly triangularizable subspace of $\Mat_n(\F)$.
\item As a consequence, if $\F$ does not have characteristic $2$ and there exists $n \geq 2$ such that $\Mat_n(\F)$ contains an irreducible optimal weakly triangularizable subspace, then $\F$ must be Pythagorean and formally real.
\end{itemize}
Remembering that if $\F$ is formally real, then every matrix of $\Mats_n(\F)$ is semi-simple\footnote{This is a simple consequence
of the fact that if we have a selfadjoint operator $u$ for a symmetric bilinear form $b$ that leaves some subspace $U$ invariant,
the orthogonal complement $U^\bot$ under $b$ is invariant under $u$: here the key is that the standard inner product
$(X,Y) \mapsto X^T Y$ is nonisotropic thanks to the assumption that $\F$ is formally real.},
we deduce that if $\F$ is not NRC and $\car(\F) \neq 2$ then $\Mats_n(\F)$ is weakly triangularizable if and only if it is weakly diagonalisable.
It is known that $\Mats_n(\F)$ is weakly diagonalisable for all $n \geq 1$
if and only if $\F$ is a formally real field and is the intersection of its real closures in a given algebraic closure
\cite{Waterhouse}, but sensible characterizations of the condition that $\Mats_n(\F)$ is weakly diagonalisable for a \emph{given} $n \geq 1$
are not known besides the easy cases where $n \in \{1,2\}$.

It remains to consider the case of NRC fields with characteristic $2$, which is quite interesting.
In section \ref{section:appendixB}, we demonstrate that there are such fields $\F$ for which $\Mats_2(\F)$ is weakly triangularizable (and we even characterize them in a simple way), but that $\Mats_3(\F)$ cannot be weakly triangularizable in such cases.

Since all fields with characteristic $2$ are Pythagorean, we deduce the following nice result
(derived from Theorems \ref{theo:generaldecomposition} and \ref{theo:optimalcarnot2}):

\begin{cor}\label{theo:nonPythagorean}
Let $\F$ be an infinite non-Pythagorean field, and $n \geq 1$ be an integer.
Then every optimal weakly triangularizable subspace of $\Mat_n(\F)$ is conjugated to $\MatT_n(\F)$.
\end{cor}

\vskip 3mm
For perfect fields with characteristic $2$, things are different. For example, over such a field every $2$-by-$2$ trace zero
matrix is triangularizable (because its characteristic polynomial reads $t^2-\alpha$ for some $\alpha \in\F$),
and hence the space $\mathfrak{sl}_2(\F)$ of all such matrices is a $3$-dimensional weakly triangularizable subspace, and an optimal one
if $\F$ is not quadratically closed. And obviously $\mathfrak{sl}_2(\F)$ is irreducible.
Interestingly, it is the only example of such an optimal irreducible space if $\F$ is infinite:

\begin{theo}\label{theo:car2}
Let $\F$ be an infinite field with characteristic $2$. Assume also that $\F$ is perfect but not quadratically closed.
Then $t_n(\F)=\dbinom{n+1}{2}$ for every integer $n \geq 1$, and the only irreducible optimal weakly triangularizable
matrix spaces over $\F$ are $\Mat_1(\F)$ and $\mathfrak{sl}_2(\F)$.
\end{theo}

Consequently, the optimal weakly triangularizable matrix spaces over a field $\F$
that satisfy the assumptions of Theorem \ref{theo:car2} are conjugates of joints of copies of
$\Mat_1(\F)$ and $\mathfrak{sl}_2(\F)$.

\subsection{Strategy, and structure of the article}

It is difficult to disprove the triangularizability of a matrix with little knowledge of the specifics of the underlying field.
Fortunately, the ``non-quadratically closed" assumption can be put to good use, but only with matrices of very small rank
(essentially, matrices of rank $2$). So, one of our main tools will of course be the characteristic polynomial
$$\chi_M=t^n-\tr(M)\, t^{n-1}+c_2(M)\, t^{n-2}-\cdots+(-1)^n \det M$$
of a matrix $M \in \Mat_n(\F)$, and more specifically the $c_2$ coefficient.
We recall that $c_2$ is a nondegenerate quadratic form on $\Mat_n(\F)$, with polar form
$$b_2 : (A,B) \mapsto \tr(A)\tr(B)-\tr(AB)=c_2(A+B)-c_2(A)-c_2(B)$$
(notice the absence of any division by $2$, so as to take fields with characteristic $2$ into account).
Remarkably, the present article will combine \emph{all} the main methods that we have discovered in the last decade to study similar issues:
\begin{itemize}
\item The adapted vectors method, that was designed to study spaces of matrices with limited number of distinct eigenvalues
\cite{dSPfeweigenvalues}.
\item The dual-orthogonality method, that was used to relate the study of weakly diagonalisable subspaces
to the one of trivial spectrum subspaces
\cite{dSPSEVdiag2}.
\item The operator-vector duality method, that has already yielded spectacular results in the study of trivial spectrum spaces
\cite{dSPAtkinsontoGerstenhaber}.
\end{itemize}

Now, we can explain the main strategy. We start with a definition:

\begin{Def}
Let $V$ be a finite-dimensional vector space, and $\calS$ be a linear subspace of $\End(V)$.
A nonzero vector $x \in V \setminus \{0\}$ is said to be \textbf{$\calS$-adapted}
when $\calS$ contains no nonzero operator with range $\F x$ and trace zero.
\end{Def}

The key point is that every weakly triangularizable subspace of $\End(V)$ has an adapted vector provided that the underlying field is NRC
(see Proposition \ref{prop:goodprop} in Section \ref{section:adaptedvectors}).
The proof of this statement is almost a word-for-word adaptation of the one of proposition 2.1 in \cite{dSPfeweigenvalues}, but since there are slight differences
and the reader might not want to look elsewhere, we have decided to include it. From there, Theorem \ref{theo:dim} is easily obtained by induction on $n$
(see Section \ref{section:fromadaptedtomajo}).
However, this strategy only works for NRC fields.

Next, we take an optimal weakly triangularizable space $\calS \subseteq \End(V)$, and we assume that it is irreducible and that $\F$ is infinite. The key is to move the problem to the orthogonal space
$\calS^\bot \subseteq \End(V)$ for the nondegenerate symmetric bilinear form
$$(u,v) \in \End(V)^2 \mapsto \tr(uv).$$
And then we introduce a third space of linear operators: the dual-operator space
$\widehat{\calS^\bot} \subseteq \Hom(\calS^\bot,V)$ that consists of all the evaluation mappings
$$\widehat{x} : u \in \calS^\bot \longmapsto u(x) \in V, \quad \text{where $x \in V$.}$$
It will be shown (Claim \ref{claim:Sbotreduced}) that $\calS^\bot$ satisfies the following weak form of irreducibility:

\begin{Def}
Let $U$ and $U'$ be vector spaces, and $\calT$ be a linear subspace of $\Hom(U,U')$.
We say that $\calT$ is:
\begin{itemize}
\item \textbf{source-reduced} when there is no $x \in U \setminus \{0\}$ such that $u(x)=0$ for all $u \in \calT$;
\item \textbf{target-reduced} when there is no linear hyperplane $H'$ of $U'$ such that $u(x) \in H'$ for all $u \in \calT$ and all $x \in U$;
\item \textbf{reduced} when it is both source-reduced and target-reduced.
\end{itemize}
\end{Def}

\begin{Not}
Let $U$ and $U'$ be finite-dimensional vector spaces, and $\calT$ be a linear subspace of $\Hom(U,U')$. We set
$$\maxrk \calT:=\max \{\rk u \mid u \in \calT\},$$
called the maximal rank in $\calT$.
\end{Not}

The optimality of $\calS$ will help us prove that $\maxrk \widehat{\calS^\bot} \leq n-1$, and
the existence of an $\calS$-adapted vector will even yield $\maxrk \widehat{\calS^\bot} = n-1$.
One difficulty is that it is not clear that $\calS^\bot$ should have \emph{trivial spectrum}
\cite{dSPlargerank,dSPAtkinsontoGerstenhaber}, meaning that none of its elements has a nonzero eigenvalue in $\F$.
However, one can show that a strengthened form of Atkinson's theorem on primitive spaces of bounded rank matrices
\cite{Atkinson,dSPLLD2} can still be used to decipher the structure of $\widehat{\calS^\bot}$.
Before we can be more explicit, an additional definition will be useful.

\begin{Def}
Let $U$ and $U'$ be finite-dimensional vector spaces, and $\calT$ be a linear subspace of $\Hom(U,U')$.
We say that $\calT$ is \textbf{target-semiprimitive} whenever it is reduced and
there is no $1$-dimensional linear subspace $D'$ of $U'$ such that $\maxrk (\calT \modu D')<\maxrk (\calT)$
where
$$\calT \modu D' :=\pi \calT=\{\pi \circ f \mid f \in \calT\}$$
for the canonical projection $\pi : U' \twoheadrightarrow U'/D'$.
\end{Def}

Beware that this definition is different from the one of semi-primitivity adopted in \cite{dSPLLD2} and \cite{dSPAtkinsontoGerstenhaber}
(which should be called ``source-semiprimitivity"), but one can easily adapt the situation by taking the transpose of the operator space under consideration.

Then:
\begin{itemize}
\item In the ``good" situation, $\widehat{\calS^\bot}$ is target-semiprimitive and our adaptation of Atkinson's theorem to target-semiprimitive operator spaces
\cite{dSPLLD2} readily yields that $\calS^\bot$ is represented in some basis by $P^{-1} \Mata_n(\F)$ for some invertible matrix $P$
(meaning that the elements of $\calS^\bot$ represent the $b$-alternating operators for some nondegenerate bilinear form $b$ on $V$); from there
one recovers that $\calS$ is represented by $\Mats_n(\F)P$ in the same basis; and it is not difficult to conclude that $\calS$ is represented by $\Mats_n(\F)$ in some basis.
\item In the ``bad" situation, $\widehat{\calS^\bot}$ is not target-semiprimitive, but one can nevertheless use the weak conclusion in Atkinson's theorem to find that there
is a high concentration of operators in $\calS^\bot$ with the same ``small" range, which allows one to recover critical information on $\calS$ and obtain its reducibility
(see the treatment of Case 2 in Section \ref{section:optimalAtkinson}). The treatment of this difficult case is one of the main innovations of the present article.
\end{itemize}

For finite fields (actually for fields with cardinality less than the dimension of the matrices),
the method is impractical because of the failure of Atkinson's theorem.
However, for such fields with characteristic other than $2$, we have to prove that $\calS$ is triangularizable, so we are inclined to think
that the result can be achieved in a foreseeable future thanks to the classical methods that use adapted vectors (but with far greater technical details than the proof given here).
In any case, finite fields with characteristic $2$ will probably be a greater challenge than the other finite fields, because of the failure
of the adapted vectors method in that case. We speculate that an adaptation of the method that would take rank $1$ idempotents into account might
help, but this will probably be more difficult than one can imagine.

So far, we have not considered perfect fields with characteristic $2$.
For those that are not quadratically closed yet infinite, the strategy is slightly different in that we do not directly try to prove the existence of an adapted vector: we prefer not to discuss things here, and refer the reader to
Section \ref{section:car2}, which must be read after the previous ones because it is highly derivative of the strategies developed in Section
\ref{section:optimal}.

\section{Adapted vectors}\label{section:adaptedvectors}

The present section is devoted to an adaptation of the adapted vectors method that was used in \cite{dSPfeweigenvalues}
to study spaces of matrices with at most two eigenvalues. As a consequence, we will obtain the validity of Theorem \ref{theo:dim}
over NRC fields. It is all the more intriguing that
these two problems (spaces of triangularizable matrices on the one hand, spaces of matrices with few eigenvalues on the other hand)
look to be polar opposites.

\subsection{The main result and why it yields Theorem \ref{theo:dim} for NRC fields}\label{section:fromadaptedtomajo}

Remember that, given a linear subspace $\calS$ of $\End(V)$ (where $V$ is a finite-dimensional vector space), a vector $x \in V \setminus \{0\}$
is called $\calS$-adapted whenever $\calS$ contains no operator with range $\F x$ and trace zero.

Here is our key result:

\begin{prop}\label{prop:goodprop}
Assume that $\F$ is an NRC field.
Let $\calS$ be a weakly triangularizable subspace of $\End(V)$, where $V$ is an $n$-dimensional vector space over $\F$, with $n>0$.
Then $\calS$ has an adapted vector.
\end{prop}

Let us immediately explain how this result yields the validity of Theorem \ref{theo:dim} for NRC fields.

So, we assume the validity of Proposition \ref{prop:goodprop}, and we prove Theorem \ref{theo:dim} for NRC fields by induction on $n$.
Note that the result is obvious if $n=1$. So, assume that $\F$ is an NRC field and let $n \geq 2$.
By Proposition \ref{prop:goodprop}, we can choose an $\calS$-adapted vector $x$.
Consider the subspace
$$\calS':=\{u \in \calS : u(x)=0\}.$$
Applying the rank theorem to $u \mapsto u(x)$ yields
$$\dim \calS \leq n+\dim \calS'.$$
Every $u \in \calS'$ induces a triangularizable endomorphism $\overline{u}$ of the quotient space $V/\F x$.
The mapping $\pi : u \in \calS' \mapsto \overline{u} \in \End(V/\F x)$ is linear, and its kernel is zero because $x$ is $\calS$-adapted.
By induction $\dim \pi(\calS') \leq \dbinom{n}{2}$, and hence
$$\dim \calS \leq n+\dim \pi(\calS')=\dbinom{n+1}{2}.$$

The remainder of the present section is devoted to the proof of Proposition \ref{prop:goodprop}.
In fact, we shall prove a more precise result on the $\calS$-adapted vectors.

\begin{Def}
Let $V$ be an $n$-dimensional vector space over $\F$. A \textbf{$2$-complex} of $V$
is an $n$-tuple $(V_i)_{1 \leq i \leq n}$ of linear subspaces of $V$ such that
$\dim V_i=\lfloor \frac{i+1}{2}\rfloor$ for all $i \in \lcro 1,n\rcro$.
\end{Def}

\begin{lemma}\label{goodlemma}
Assume that $\F$ is an NRC field.
Let $V$ be an $n$-dimensional vector space over $\F$, and $\calS$
be a weakly triangularizable subspace of $\End(V)$.
Then the union of a $2$-complex of $V$ never contains all the $\calS$-adapted vectors.
\end{lemma}

Of course, Lemma \ref{goodlemma} yields the existence of an $\calS$-adapted vector under the stated assumptions.
Before we can prove Lemma \ref{goodlemma}, we need several basic lemmas.

\begin{lemma}\label{tracelemma}
Assume that $\F$ is an NRC field. Let $\calM$ be a weakly triangularizable subspace of $\Mat_n(\F)$, with $n \geq 2$, and $A$ and $B$ be two matrices in
$\calM$ such that $\rk A \leq 1$, $\rk B \leq 1$ and $\tr(A)=\tr(B)=0$. Then $\tr(AB)=0$.
\end{lemma}

\begin{proof}
Let $\alpha \in \F \setminus \{0\}$.
As $\rk(A+\alpha\,B) \leq 2$ and $\tr(A+\alpha B)=0$, the characteristic polynomial of $A+\alpha\,B$ simplifies as
$$t^n+c_2(A+\alpha B)\,t^{n-2}=t^{n-2}\left(t^2+c_2(A+\alpha B)\right).$$
Moreover
$$c_2(A+\alpha B)=c_2(A)+\alpha^2 c_2(B)+\alpha (\tr(A)\tr(B)-\tr(AB))=-\alpha \tr(AB).$$
Hence $\alpha \tr(AB)$ is a square in $\F$ for all $\alpha \in \F$. And since $\F$ is an NRC field we deduce that $\tr(AB)=0$.
\end{proof}

\begin{cor}\label{majdim1star}
Assume that $\F$ is an NRC field.
Let $\calM$ be a weakly triangularizable subspace of $\Mat_n(\F)$ which is spanned by matrices of rank $1$ and trace $0$.
Then $\dim \calM \leq \dfrac{n^2}{2}\cdot$
\end{cor}

\begin{proof}
Indeed, Lemma \ref{tracelemma} shows that $\calM$ is totally singular for the non-degenerate symmetric bilinear form
$(A,B) \mapsto \tr(AB)$ on $\Mat_n(\F)$. Therefore, $\dim \calM \leq \frac{1}{2}\,\dim \Mat_n(\F)=\frac{n^2}{2}\cdot$
\end{proof}

We finally recall the following lemma from \cite{dSPfeweigenvalues} (see corollary 2.6 there):

\begin{lemma}\label{coveringlemma2}
Let $p$ be a positive integer such that $|\F|>p$.
Let $V$ be an $n$-dimensional vector space over $\F$, and $(V_i)_{i \in I}$
be a finite family of linear subspaces of $V$ in which $|I|=(n-1)p$ and, for
all $k \in \lcro 1,n-1\rcro$, exactly $p$ vector spaces in the family have dimension $k$.
Then there exists a basis of $V$ in which no vector belongs to $\underset{i \in I}{\cup} V_i$.
\end{lemma}

Note that Lemma \ref{coveringlemma2} is obvious if $\F$ is infinite, as in that case
no finite union of proper linear subspaces covers the full space.

\subsection{Proof of Lemma \ref{goodlemma}}\label{goodlemmaproof}

We use an induction on $n$. Throughout, we assume that $\F$ is NRC. Assume first that $n=2$. Let $V$ be a $2$-dimensional vector space over $\F$, and
$\calS$ be a weakly triangularizable subspace of $\End(V)$.

Assume that there are linearly independent $\calS$-inadapted vectors $x$ and $y$.
Denote by $\calM$ the subspace of $\Mat_2(\F)$ associated with $\calS$ in the basis $(x,y)$.
Then, our assumptions show that $\calM$ contains the unit matrices $E_{1,2}$ and $E_{2,1}$.
But this would contradict Lemma \ref{tracelemma}. Hence the $\calS$-inadapted vectors are trapped inside a $1$-dimensional subspace of $V$.

Because $|\F|>2$ there are at least four $1$-dimensional subspaces of $V$, and hence the set of all
$\calS$-adapted vectors is not included in the union of two such subspaces.
Hence the case $n=2$ in Lemma \ref{goodlemma} is proved.

Let now $n \geq 3$, and assume that the result of Lemma \ref{goodlemma} holds for $n-1$.
Let $V$ be an $n$-dimensional vector space, and $\calS$ be a weakly triangularizable subspace of $\End(V)$.
We use a \emph{reductio ad absurdum} by assuming that there exists a $2$-complex $(V_1,\dots,V_n)$ of $V$
whose union contains every $\calS$-adapted vector.

\begin{claim}\label{gooddualbasis}
There exists a basis $(\varphi_1,\dots,\varphi_n)$
of the dual space $V^\star$ such that no $V_i$ is included in some $\Ker \varphi_j$.
\end{claim}

\begin{proof}
Set $W^o:=\{f \in V^\star : \forall x \in W, \; f(x)=0\}$ for any linear subspace $W$ of $V$.

As the subspaces $(V_1)^o,\dots,(V_n)^o$ have their dimensions respectively equal to
$n-1,n-1,n-2,n-2,\dots,n-\lfloor \frac{n+1}{2}\rfloor$, and as $|\F|>2$ (because $\F$ is NRC)
Lemma \ref{coveringlemma2} yields a basis of $V^\star$ made of vectors that belong to none of the spaces $(V_i)^o$. This
yields the claimed result.
\end{proof}

Given a (linear) hyperplane $H$ of $V$, we define the linear subspace
$$\calS_H:=\Bigl\{u \in \calS : \; \tr(u)=0 \quad \text{and} \quad \forall x \in H, \; u(x)=0\Bigr\}.$$

\begin{claim}\label{hyperplaneclaim}
Let $H$ be a linear hyperplane of $V$ that includes none of the spaces $V_1,\dots,V_n$.
Then $\dim \calS_H>\lfloor \frac{n}{2}\rfloor$.
\end{claim}

\begin{proof}
Set $p:=\lfloor \frac{n}{2}\rfloor$. Denote by $\calT$ the linear subspace of $\calS$ consisting of its elements with range included in $H$,
and by $\calT' \subseteq \End(H)$ the space of endomorphisms induced by the elements of $\calT$.
Then $\calT'$ is weakly triangularizable.
We contend that the set of all vectors of $H$ that are both $\calS$-inadapted and $\calT'$-adapted
has its rank greater than $p$.

Assume that this is not true.
Then some $p$-dimensional linear subspace $G$ of $H$
contains all the vectors of $H$ that are both $\calS$-inadapted and $\calT'$-adapted.
As $H$ includes none of $V_1,\dots,V_n$, one sees that $V_1 \cap H=V_2 \cap H=\{0\}$ and
that $\bigl(V_3 \cap H,\dots,V_n \cap H,G\bigr)$ is a $2$-complex of $H$.
By the induction hypothesis, at least one $\calT'$-adapted vector belongs to none of $G,V_1,V_2,\dots,V_n$.
Yet, such a vector must be $\calS$-adapted since it belongs to $H$ but not to $G$.
This is in contradiction with our initial assumption, which stated that
no vector outside of $V_1 \cup \cdots \cup V_n$ is $\calS$-adapted.

Hence we can find $p+1$ linearly independent vectors $x_1,\dots,x_{p+1}$ of $H$
that are all $\calS$-inadapted yet $\calT'$-adapted. Let us extend this family into a basis $(x_1,\dots,x_{n-1})$ of $H$,
and then into a basis $(x_1,\dots,x_n)$ of $V$. Consider the matrix space $\calM \subseteq \Mat_n(\F)$
associated with $\calS$ in that basis. Let $i \in \lcro 1,p+1\rcro$. Since $x_i$ is $\calS$-inadapted,
there is a non-zero matrix $A \in \calM$ with trace zero and all rows zero except the $i$-th one.
Write
$$A=\begin{bmatrix}
B & [?]_{(n-1) \times 1} \\
[0]_{1 \times (n-1)} & 0
\end{bmatrix} \quad \text{with $B \in \Mat_{n-1}(\F)$,}$$
and note that $B$ represents an element of $\calT'$ in the basis $(x_1,\dots,x_{n-1})$.
Note also that $\tr B=\tr A=0$ and that $B$ has all rows zero with the possible exception of the $i$-th. As $x_i$ is $\calT'$-adapted,
we conclude that $B=0$. It follows that $A$ is a non-zero scalar multiple of $E_{i,n}$.
Therefore, the linear subspace $\calM$ includes $\Vect(E_{1,n},\dots,E_{p+1,n})$. As $(x_1,\dots,x_{n-1})$ is a basis of $H$,
all the matrices of $\Vect(E_{1,n},\dots,E_{p+1,n})$ represent elements of $\calS_H$ in the basis $(x_1,\dots,x_n)$.
This yields $\dim \calS_H >p$.
\end{proof}

We are ready to conclude. Let us choose a basis $(\varphi_1,\dots,\varphi_n)$ of $V^\star$
that satisfies the conclusion of Claim \ref{gooddualbasis}. For $i \in \lcro 1,n\rcro$, set $\calS_i:=\calS_{\Ker \varphi_i}$.
Clearly, the linear subspaces $\calS_1,\dots,\calS_n$ of $\End(V)$ are linearly independent.
Finally, setting
$$\calS':=\calS_1 \oplus \cdots \oplus \calS_n,$$
we deduce from Claim \ref{hyperplaneclaim} that
$$\dim \calS' \geq n\,\Bigl(1+\Bigl\lfloor \frac{n}{2}\Bigr\rfloor\Bigr)>\frac{n^2}{2}\cdot$$
This contradicts Corollary \ref{majdim1star} since $\calS'$ is included in the weakly triangularizable space $\calS$
and is spanned by elements of rank $1$ and trace $0$.
We deduce that our initial assumption was false, i.e.\ the union of a $2$-complex of $V$ cannot contain
all the $\calS$-adapted vectors.

This finishes our proof of Lemma \ref{goodlemma} by induction. Proposition \ref{prop:goodprop} follows directly from it.
And, as explained in Section \ref{section:fromadaptedtomajo}, Theorem \ref{theo:dim} follows from Proposition \ref{prop:goodprop}
for NRC fields. For perfect infinite fields with characteristic $2$, we refer the reader to Section \ref{section:car2}.

\section{Optimal spaces for NRC fields: infinite fields}\label{section:optimal}

\subsection{The Erasure Lemma}

We start with a simple lemma, which will be used at key points in our proof of Theorem \ref{theo:optimalcarnot2}.
Note that it holds over any field that is not quadratically closed, and does not require that $\car(\F) \neq 2$.

\begin{lemma}[Erasure Lemma]\label{lemma:erasure1}
Assume that $\F$ is not quadratically closed.
Let $n \geq 1$ be an integer.
Let $N \in \Mat_n(\F)$ and $C \in \F^n$. Assume that, for every row
$R \in \Mat_{1,n}(\F)$ and every $a \in \F$, the matrix
$$\begin{bmatrix}
a & R \\
C & N
\end{bmatrix}$$
is triangularizable. Then $C=0$.
\end{lemma}

The proof of Lemma \ref{lemma:erasure1} involves the following classical result in the study of cyclic matrices
(see e.g.\ lemma 11 of \cite{dSPidempotentLC}). In it, remember that the trace of a monic polynomial of degree $d \geq 1$
is defined as the opposite of its coefficient on $t^{d-1}$.

\begin{lemma}\label{lemma:adaptationcyclic}
Let $M=(m_{i,j}) \in \Mat_n(\F)$ be a regular Hessenberg matrix, i.e.
$m_{i,j}= 0$ for all $i,j$ such that $i>j+1$, and $m_{j+1,j} \neq 0$ for all $j \in \lcro 1,n-1\rcro$.
Let $r(t) \in \F[t]$ be a monic polynomial of degree $n$ with $\tr (r)=\tr (M)$.
Then there exists a row matrix $R \in \Mat_{1,n-1}(\F)$ such that the matrix
$$M+\begin{bmatrix}
0 & R \\
[0]_{(n-1) \times 1} & [0]_{(n-1) \times (n-1)}
\end{bmatrix}$$
has characteristic polynomial $r(t)$.
\end{lemma}

\begin{proof}[Proof of Lemma \ref{lemma:erasure1}]
Assume on the contrary that $C \neq 0$.
Consider $M:=\begin{bmatrix}
0 & [0]_{1 \times n} \\
C & N
\end{bmatrix} \in \Mat_{n+1}(\F)$.
We change the basis so as not to change the first vector $x$, and
take a basis in which $e_1=x$, $e_2=Mx$, $e_3=M^2 x$ and so on, so that $(e_1,\dots,e_d)$ spans $\F[M]x$ for some $d \in \lcro 2,n\rcro$.
This change of basis does not affect that starting assumption, but now $M$ has the reduced shape
$$M=\begin{bmatrix}
C(p) & [?]_{d \times (n+1-d)} \\
[0]_{(n+1-d) \times d} & [?]_{(n+1-d) \times (n+1-d)}
\end{bmatrix}$$
where $C(p)$ denotes the companion matrix of some monic polynomial $p(t)$ with degree $d \geq 2$.
We deduce that for every row $R' \in \Mat_{1,d-1}(\F)$ and every $a \in \F$, the matrix
$$C(p)+\begin{bmatrix}
a & R' \\
[0]_{(d-1) \times 1} & [0]_{(d-1) \times (d-1)}
\end{bmatrix}$$
has its characteristic polynomial split over $\F$. By Lemma \ref{lemma:adaptationcyclic}, in varying $a$ and $R'$
the characteristic polynomials we obtain are exactly the monic polynomials of degree $d$ of $\F[t]$.
And because $\F$ is not quadratically closed one of these polynomials does not split over $\F$
(it suffices to take the product $t^{d-2} q(t)$, where $q \in \F[t]$ is monic with degree $2$ and has no root in $\F$).
Therefore, we have a contradiction, and we conclude that $C=0$.
\end{proof}

\subsection{Setting the proof up}

Throughout the present section, we let $\F$ be an infinite field.
We assume furthermore that $\F$ is NRC.
We let $V$ be an $n$-dimensional vector space over $\F$, and
$\calS$ be an irreducible optimal weakly triangularizable linear subspace of $\End(V)$.
By Theorem \ref{theo:dim} we know that $\dim \calS=\frac{n(n+1)}{2}\cdot$
Our aim is to prove that there exists a basis of $V$
in which the elements of $\calS$ are represented by the $n$-by-$n$ symmetric matrices over $\F$.
The case $n=1$ is obvious, so we assume that $n>1$ (note that the proof is inductive, but this will be apparent only at its very last step).

As explained earlier, a key part will be played by the trace-orthogonal complement
$$\calS^\bot:=\{u \in \End(V) : \forall v \in \calS, \; \tr(uv)=0\}$$
for the nondegenerate symmetric bilinear form $(u,v) \mapsto \tr(uv)$ on $\End(V)$.
And further, we consider the dual-operator space $\widehat{\calS^\bot}$ of $\calS^\bot$, made of all the evaluation mappings
$$\widehat{x} : u \in \calS^\bot \longmapsto u(x) \in V, \quad \text{with $x \in V$.}$$
Here is a key point, which connects the problem of $\calS$-adapted vectors to rank issues in the space $\widehat{\calS^\bot}$.

\begin{lemma}\label{lemma:adaptedtodualop}
Let $x \in V \setminus \{0\}$. Then
$$\rk \widehat{x}=n-\dim (\calS \cap \Hom(V,\F x)).$$
\end{lemma}

This lemma is a special case of a more general result, which will be reused later in the article:

\begin{lemma}\label{lemma:orthocomplement}
Let $W$ be a linear subspace of $V$ and $\calT$ be an arbitrary linear subspace of $\End(V)$. Then
$$\dim (\calT \cap \Hom(V,W))+\dim \{u_{|W} \mid u \in \calT^\bot\}=(\dim W)\,(\dim V).$$
\end{lemma}

Lemma \ref{lemma:adaptedtodualop} is obtained from Lemma \ref{lemma:orthocomplement} by taking $W=\F x$ and by noting that
the range of $\widehat{x}$ is isomorphic to $\{u_{|\F x} \mid u \in \calT^\bot\}$.

\begin{proof}[Proof of Lemma \ref{lemma:orthocomplement}]
Denote by $\pi$ the canonical projection from $V$ onto $V/W$.
For all $v \in \calT \cap \Hom(V,W)$ and all $u \in \calT^\bot$, we note that
$v \circ u$ maps into $W$, and hence $0=\tr(v \circ u)=\tr((v \circ u)_{|W})=\tr(v \circ (u_{|W}))$.
Hence $\{u_{|W} \mid u \in \calT^\bot\}$ is right-orthogonal to $\calT \cap \Hom(V,W)$ for the non-degenerate bilinear form
$$(v,u) \in \Hom(V,W) \times \Hom(W,V) \longmapsto \tr(v\circ u),$$
to the effect that
\begin{equation}\label{eq:ortho1}
\dim (\calT \cap \Hom(V,W))+\dim \{u_{|W} \mid u \in \calT^\bot\} \leq (\dim W)\,(\dim V).
\end{equation}

Next, for all $v \in \calT$ and all $u \in \calT^\bot$ such that $u_{|W}=0$, we find that $v \circ u$ vanishes everywhere on $W$, and hence
$$0=\tr(v \circ u)=\tr((\pi \circ v) \circ \overline{u})$$
where $\overline{u} : V/W \rightarrow V$ denotes the linear mapping induced by $u$.
Hence $\pi \calT:=\{\pi \circ v \ mid v \in S\}$ is left-orthogonal to
$$\calT':=\{\overline{u} \mid u \in \calT^\bot \; \text{such that}\; u_{|W}=0\}$$
for the non-degenerate bilinear form
$$(v,u) \in \Hom(V,V/W) \times \Hom(V/W,V) \longmapsto \tr(v\circ u).$$
Hence
\begin{equation}\label{eq:ortho2}
\dim (\pi \calT)+\dim \calT' \leq (\dim (V/W))\,(\dim V).
\end{equation}
Summing \eqref{eq:ortho1} and \eqref{eq:ortho2} yields
$$\dim (\calT \cap \Hom(V,W))+\dim (\pi \calT)+\dim \{u_{|W} \mid u \in \calT^\bot\}+\dim \calT' \leq (\dim V)^2=\dim \End(V).$$
Yet by the rank theorem $\dim \calT=\dim (\calT \cap \Hom(V,W))+\dim (\pi \calT)$ and
$\dim \calT^\bot=\dim \{u_{|W} \mid u \in \calT^\bot\}+\dim \calT'$.
Hence there is equality in \eqref{eq:ortho1}, which completes the proof.
\end{proof}

Now, we move on to our first step.

\begin{claim}\label{claim:vectorx}
One has $\rk\widehat{x} <n$ for all $x \in V \setminus \{0\}$.
\end{claim}

\begin{proof}
Let $x \in V \setminus \{0\}$, and assume that $\rk \widehat{x}=n$.
Then by Lemma \ref{lemma:adaptedtodualop} the space $\calS$ contains no operator with range $\F x$.
Here (and this is slightly different from what we did in Section \ref{section:fromadaptedtomajo}), we
consider the standard projection $\pi : V \twoheadrightarrow V/\F x$,
the kernel $\calS'$ of
$$u \in \calS \mapsto \pi(u(x)) \in V/\F x,$$
and then the projection mapping
$$u \in \calS' \mapsto \overline{u} \in \End(V/\F x).$$
We have just seen that the latter is injective, and obviously its range is a weakly triangularizable subspace of $\End(V/\F x)$.
By Theorem \ref{theo:dim}, we deduce that $\dim \calS' \leq \dbinom{n}{2}$.
Hence
$$\dim \calS \leq \dim(V/\F x)+\dim \calS' \leq (n-1)+\dbinom{n}{2}=\dbinom{n+1}{2}-1,$$
thereby contradicting the optimality of $\calS$.
\end{proof}

\begin{claim}\label{claim:Sbotreduced}
The space $\calS^\bot$ is reduced.
\end{claim}

\begin{proof}
Assume first that there is a nonzero vector $x \in V$ in the kernel of all the elements of $\calS^\bot$.
Hence $\calS=(\calS^\bot)^\bot$ includes $\Hom(V,\F x)$
(by Lemma \ref{lemma:orthocomplement} applied to $W=\F x$).

Let us extend $x$ into a basis $(x,e_2,\dots,e_n)$ of $V$, and then consider the matrix space $\calM$
associated with $\calS$ in that basis.
Then $\calM$ contains every matrix of the form
$\begin{bmatrix}
a & R \\
[0]_{(n-1) \times 1} & [0]_{(n-1) \times (n-1)}
\end{bmatrix}$ with $a \in \F$ and $R \in \Mat_{1,n-1}(\F)$.
By the Erasure Lemma,
every matrix of $\calM$ then reads
$\begin{bmatrix}
? & [?]_{1 \times (n-1)} \\
[0]_{(n-1) \times 1} & [?]_{(n-1) \times (n-1)}
\end{bmatrix}$, which means that $\F x$ is invariant under all the elements of $\calS$.
This contradicts the irreducibility of $\calS$.

We have just proved that $\calS^\bot$ is source-reduced. To obtain that $\calS^\bot$ is also target-reduced, we
simply apply the first step to the transposed operator space $\calS^t :=\{f \in V^\star \mapsto f \circ u \mid u \in \calS\} \subseteq \End(V^\star)$, which is
still an irreducible optimal weakly triangularizable subspace.
\end{proof}

\subsection{Applying Atkinson's theorem}\label{section:optimalAtkinson}

At this point, we have extracted as much information as we could from $\calS$. Now our analysis moves to $\calS^\bot$, through its dual operator space
$\widehat{\calS^\bot}$.
As seen in Claim \ref{claim:vectorx}, every element of $\widehat{\calS^\bot}$ has rank at most $n-1$.
Moreover, since $\calS^\bot$ is target-reduced it is clear that $\widehat{\calS^\bot}$ is also target-reduced.
And finally $\widehat{\calS^\bot}$ is source-reduced (as is every dual operator space). Hence $\widehat{\calS^\bot}$ is reduced.
Beware however that $\widehat{\calS^\bot}$ might not be target-semiprimitive.
But if it is we will be helped by the following version of Atkinson's theorem \cite{Atkinson} on target-semiprimitive spaces of bounded rank operators:

\begin{theo}[Atkinson's theorem]\label{theo:Atkinson}
Let $\calT$ be a vector space of linear operators between finite-dimensional vector spaces $U_1$ to $U_2$
such that $\rk f<\dim U_2$ for all $f \in \calT$. Set $n:=\dim U_2$.
Assume that $\calT$ is target-semiprimitive and that $|\F| \geq n$.
Then:
\begin{enumerate}[(a)]
\item $\dim U_1 \leq \dbinom{n}{2}$.
\item If an addition $\dim U_1>\dbinom{n-1}{2}+1$,
then there exists an \textbf{alternator pair} for $\calT$, i.e.\ a pair $(b,\varphi)$ consisting of a nondegenerate
bilinear form $b : U_2 \times U_2 \rightarrow \F$ and of a vector space isomorphism
$\varphi : \calT \overset{\simeq}{\longrightarrow} U_2$ such that
$$\forall x \in U_1, \; \forall f \in \calT, \; b(\varphi(f),f(x))=0;$$
moreover $\maxrk (\calT)=n-1$ in that case.
\end{enumerate}
\end{theo}

This theorem is a transposed version of theorem 5.3 of \cite{dSPLLD2}. Beware that Atkinson's formulation in \cite{Atkinson}
is also a transposed analogue of this theorem, due to the fact that Atkinson sees matrices as acting on rows rather than on columns
(so, in his article the number of rows corresponds to the dimension of the source space).
Some explanation is then needed regarding the conclusion of point (b). To obtain it, we apply the formulation from \cite{dSPLLD2}
to the transposed operator space $\calT^t:=\{f^t \mid f \in \calT\} \subseteq \Hom(U_2^\star,U_1^\star)$, thereby recovering
a linear mapping $\Phi : U_2^\star \wedge U_2^\star \rightarrow U_1^\star$ such that $\Psi : g \in U_2^\star \mapsto \Phi(g \wedge (-))$
is an isomorphism from $U_2^\star$ to $\calT^t$. Then, we choose an arbitrary nondegenerate bilinear form
$b : U_2 \times U_2 \rightarrow \F$, allowing us to identify $U_2^\star$ with $U_2$ through $\Theta : z \in U_2 \mapsto b(z,-) \in U_2^\star$, and
we set $\varphi(f):=\Theta^{-1} (\Psi^{-1} (f^t))$ for $f \in \calT$, thereby recovering a vector space isomorphism
$\varphi : \calT \overset{\simeq}{\longrightarrow} U_2$. We shall now prove the identity
$$\forall x \in U_1, \; \forall f \in \calT, \; b(\varphi(f),f(x))=0.$$
For this, let $f \in \calT$. Then $z:=\varphi(f)$ is the unique vector of $U_2$ such that $\Psi(b(z,-))=f^t$.
Hence $\Phi(b(z,-) \wedge e)=f^t(e)=e \circ f$ for all $e \in U_2^\star$. This applies in particular to $e=b(z,-)$,
yielding $b(z,f(x))=0$ for all $x \in U_1$, as claimed.

It is important to stress that the bilinear form $b$ has no special property besides being nondegenerate (in fact, by changing the isomorphism $\varphi$
is is clear that we can take an arbitrary nondegenerate bilinear form on $U_2$).

Let us now apply this theorem to $\calT:=\widehat{\calS^\bot}$, knowing that here $U_1:=\calS^\bot$ and $U_2=V$, which have respective dimensions
$\dbinom{n}{2}$ and $n$. Note that point (b) can be applied only if $\dbinom{n}{2}>\dbinom{n-1}{2}+1$, i.e.\ $n>2$.

\noindent \textbf{Case 1: $\calT$ is target-semiprimitive, or $n=2$}. \\
The conclusion is easy in that case.
Assume first that $n>2$. Then by Atkinson's theorem we can find an alternator pair $(b,\varphi)$ for
$\widehat{\calS^\bot}$. Then
$$\forall u \in \calS^\bot, \; \forall x \in V, \;  b(\varphi(\widehat{x}),u(x))=0.$$
Hence the bilinear form $c : (x,y)\in V^2 \mapsto b(\varphi(\widehat{x}),y)$, which is clearly nondegenerate, is an \emph{alternator} for $\calS^\bot$,
meaning that every $u \in \calS^\bot$ is $c$-alternating (i.e.\ $c(x,u(x))=0$ for all $x \in V$).

Assume now that $n=2$. Then $\dim \calS^\bot=1$, and we take an arbitrary $u_0 \in \calS^\bot \setminus \{0\}$. Since $\calS^\bot$
is reduced we see that $u_0$ is an automorphism of $V$.
We take an arbitrary nondegenerate alternating bilinear form $b$ on $V$ (which is possible because $n$ is even).
Since $u_0$ is an automorphism the bilinear form $c : (x,y) \mapsto b(x,u_0^{-1}(y))$ is nondegenerate, and as $b$ is alternating
we find that $c$ is alternator for $\calS^\bot$.

Hence, in any case we have found a nondegenerate alternator $c$ for $\calS^\bot$.
Since $\calS^\bot$ has dimension $\dbinom{n}{2}$ it must be equal to the space $\calA_c$ of all $c$-alternating endomorphisms of $V$.

Now, consider a matrix space $\calM$ associated with $\calS$ in some basis of $V$, and the matrix $P \in \GL_n(\F)$ that represents the bilinear form $c$
in that basis (i.e.\ its Gram matrix). By the above
$\calM^\bot=P^{-1}\Mata_n(\F)$, and consequently
$$\calM=\Mats_n(\F)P.$$

Now, note that $I_n \in \calM$ because clearly $\F I_n+\calM$ is weakly triangularizable, and
$\calM$ is optimal. Hence $P^{-1}$ is symmetric, i.e.\ $P$ is symmetric.

Next, for all $R \in \GL_n(\F)$ we note that
$$\Mats_n(\F)R^T PR=R^{-1}(R\Mats_n(\F)R^T) PR=R^{-1}(\Mats_n(\F)P)R.$$
Hence replacing $P$ with a congruent matrix simply replaces $\calM$ with a similar matrix space.

At this point, we must remember that $\F$ is allowed to have characteristic $2$, and split the discussion into two cases.
Either $P$ is alternating (and $\car(\F)=2$) or $P$ is congruent to a diagonal matrix
(see e.g.\ theorem 3.0.13 in \cite{dSPquadform}).

Assume first that $P$ is congruent to a diagonal matrix. Without loss of generality we can assume that $P$ is diagonal.
Denote then by $p_1,\dots,p_n$ the diagonal entries of $P$.
Let $i \in \lcro 2,n\rcro$. Then $(E_{1,i}+E_{i,1}) P= p_i E_{1,i}+ p_1 E_{i,1}$ is triangularizable, with characteristic polynomial
$t^{n-2}(t^2-p_1p_i)$. Hence $p_1p_i$ is a square in $\F$, and therefore $p_i=\alpha_i^2 p_1$ for some $\alpha_i \in \F^\times$.
We conclude that $P$ is congruent to $p_1 I_n$, and hence $\calS$ is represented by $p_1 \Mats_n(\F)=\Mats_n(\F)$ in some basis of $V$.

Assume finally that $P$ is alternating and $\car(\F)=2$. We will show that this leads to a contradiction.
First of all $n=2m$ for some $m \geq 1$, and by replacing $P$ with a congruent matrix we can assume that
$$P=\begin{bmatrix}
[0]_{m \times m} & I_m \\
-I_m & [0]_{m \times m}
\end{bmatrix}.$$
\begin{itemize}
\item Assume first that $m>1$. Then, for all $A \in \Mat_m(\F)$, the matrix
$$M=\begin{bmatrix}
A & [0]_{m \times m} \\
[0]_{m \times m} & -A^T
\end{bmatrix}$$
is such that $MP^{-1}$ is symmetric, and hence by the above $M$ must be triangularizable over $\F$.
It would follow that every matrix $A \in \Mat_m(\F)$ is triangularizable over $\F$: this is absurd because $m \geq 2$ and
$\F$ is not quadratically closed.
\item Assume finally that $m=1$. Then, for all $\lambda \in \F$ the matrix
$\begin{bmatrix}
0 & 1 \\
\lambda & 0
\end{bmatrix}$ must belong to $\calM$, and hence its characteristic polynomial $t^2-\lambda$ splits over $\F$. This contradicts the assumption that $\F$ is NRC.
\end{itemize}

The proof is complete in the target-semiprimitive case.

\vskip 3mm
\noindent \textbf{Case 2: $\calT$ is not target-semiprimitive}. \\
Then, we take a maximal subspace $\{0\} \subsetneq W\subsetneq V$ such that, for the canonical projection $\pi : V \twoheadrightarrow V/W$, all
the composite operators $\pi\widehat{x}$ have rank at most $\maxrk (\calT)-\dim W$, and in particular are nonsurjective.
We will ultimately show that $W$ is $\calS$-invariant, but it is far from obvious at this point.

The resulting operator space $\pi \calT$ remains target-reduced because $\calT$ is target-reduced, but it might not be source-reduced.
So we introduce the intersection $\calH$ of all kernels of the operators in $\pi \calT$, and we take a direct factor $\calH'$
of $\calH$ in $\calS^\bot$, so that
$$\calH \oplus \calH'=\calS^\bot.$$
The resulting operator space $\pi \calT_{|\calH'}$ is now target-semiprimitive, owing to the maximality of $W$, so
we can apply the first conclusion of Atkinson's theorem to derive that
$$\dim \calH' \leq \dbinom{s}{2}, \quad \text{where} \;  s:=\dim(V/W).$$
Now, we look at how this plays out in $\calS$ and $\calS^\bot$.
We readily have
$$\dim \calH =\dim \calS^\bot-\dim \calH' \geq \dbinom{n-s}{2}+s(n-s).$$
Moreover, the very definition of $\calH$ shows that all its operators have their range included in $W$.
By Lemma \ref{lemma:orthocomplement} applied to $\calS^\bot$ (using $\calS^{\bot\bot}=\calS$), this shows that the operator space
$$\calS_W:=\{u_{|W} \mid u \in \calS\} \subseteq \Hom(W,V)$$
satisfies
$$\dim \calS_W \leq \dim(\Hom(W,V))-\dim \calH \leq \dbinom{n-s+1}{2}.$$
This prompts us to consider the kernel
$$\calS':=\{u \in \calS : \; u_{|W}=0\}.$$
By the rank theorem
$$\dim \calS' =\dim \calS-\dim \calS_W \geq \dim \calS-\dbinom{n-s+1}{2}=s(n-s)+\dbinom{s+1}{2}.$$
Now, we take a basis $\bfB$ of $V$ in which the first $n-s$ vectors span $W$,
and we consider the matrix spaces $\calM$ and $\calM'$ that correspond to $\calS$ and $\calS'$, respectively, in those bases.
Hence every $u \in \calS'$ has its representing matrix, in $\bfB$, of the form
$$M(u)=\begin{bmatrix}
[0]_{(n-s) \times (n-s)} & C(u) \\
[0]_{s \times (n-s)} & D(u)
\end{bmatrix} \quad \text{where $D(u) \in \Mat_s(\F)$ and $C(u) \in \Mat_{n-s,s}(\F)$.}$$
Then $D(\calS')$ is a weakly triangularizable subspace of $\Mat_s(\F)$, and hence by Theorem \ref{theo:dim}:
$$\dim D(\calS') \leq \dbinom{s+1}{2}.$$
As trivially $\dim C(\calS') \leq s(n-s)$, we obtain
$$\dim \calS' \leq \dim C(\calS')+\dim D(\calS') \leq s(n-s)+\dbinom{s+1}{2}.$$
Since we knew that $\dim \calS'\geq \dbinom{s+1}{2}+s(n-s)$, we derive that
$\dim D(\calS')=\dbinom{s+1}{2}$ and that
$\calM'$ contains every matrix of the form
$$\begin{bmatrix}
[0]_{(n-s) \times (n-s)} & C' \\
[0]_{s \times (n-s)} & D'
\end{bmatrix} \quad \text{with $C' \in \Mat_{n-s,s}(\F)$ and $D' \in D(\calS')$.}$$
The conclusion is nearby.

First of all, we take an arbitrary matrix $M \in \calM$. By adding to $M$ an appropriate matrix of the previous kind, we
obtain a new matrix
$$M'=\begin{bmatrix}
[?]_{(n-s) \times (n-s)} & [0]_{(n-s) \times s} \\
[?]_{s \times (n-s)} & D_1
\end{bmatrix} \in \calM, \qquad \text{where $D_1 \in \Mat_s(\F).$}$$
Since $\lambda M'+M''$ is triangularizable for all $\lambda \in \F$ and all $M'' \in \calM'$, we find that
$\lambda D_1+D'$ is triangularizable for all $\lambda \in \F$ and all $D' \in D(\calS')$. However $\dim D(\calS')=t_s(\F)$,
and hence $D_1 \in D(\calS')$. Therefore every matrix of $\calM$ has its lower-right block in $D(\calS')$.

It is now time to apply the inductive assumption to $D(\calS')$. This assumption, combined with Theorem \ref{theo:generaldecomposition},
yields a nontrivial partition $s=d_1+\dots+d_p$ such that
$D(\calS')$ is conjugated to $\Mats_{d_1}(\F) \vee \cdots \vee \Mats_{d_p}(\F)$.
Hence, by modifying the last $s$ vectors of the basis $\bfB$ we can assume that
$$D(\calS')=\Mats_{d_1}(\F) \vee \cdots \vee \Mats_{d_p}(\F).$$
In particular -- and this is the only feature we will use from the structure of $D(\calS')$ -- we now have that
$D(\calS')$ contains the unit matrix $E_{i,i}$ for all $i \in \lcro 1,s\rcro$.

Finally, take $M$ in $\calM$. Then, for some $N \in \calM'$, we have
$$M-N=\begin{bmatrix}
A & [0]_{(n-s) \times s} \\
B & [0]_{s \times s}
\end{bmatrix} \qquad \text{where $A \in \Mat_{n-s}(\F)$ and $B \in \Mat_{s,n-s}(\F)$.}$$
Now, let $C \in \F^{n-s}$ and $a \in \F$. Writing
$B=\begin{bmatrix}
R_1 \\
B_1
\end{bmatrix}$ with $R_1 \in \Mat_{1,n-s}(\F)$ and $B_1 \in \Mat_{s-1,n-s}(\F)$, we use the above shape of $\calM'$
and the fact that $E_{1,1}$ belongs to $D(\calS')$ to obtain that
$$\begin{bmatrix}
A & C & [0]_{(n-s) \times (s-1)} \\
R_1 & a & [0]_{1 \times (s-1)} \\
B_1 & [0]_{(s-1) \times 1} & [0]_{(s-1) \times (s-1)}
\end{bmatrix} \in \calM.$$
We deduce that
$\begin{bmatrix}
A & C \\
R_1 & a
\end{bmatrix}$ is triangularizable over $\F$, and by permuting the basis vectors and transposing we conclude that
$\begin{bmatrix}
a & C^T \\
R_1^T & A^T
\end{bmatrix}$ is triangularizable over $\F$.
Hence the Erasure Lemma yields $R_1^T=0$, i.e.\ $R_1=0$.

Likewise we obtain that all the rows of $B$ are zero. Hence, we have shown that $\calS$ is reducible
(and more precisely that $W$ is invariant under all the elements of $\calS$), which contradicts our starting assumption.
Hence the inductive step has been climbed, and our proof of Theorem \ref{theo:optimalcarnot2} is now complete.

\begin{Rem}
An alternative route, in the last segment of the proof, consists in applying point (b) of Atkinson's theorem to the space
$\pi \calT_{|\calH'}$ to prove directly that $D(\calS') \simeq \Mats_s(\F)$
(with the case $s=2$ requiring a similar treatment as we have done for Case 1).
This would allow us to avoid using an induction hypothesis to understand the structure of $D(\calS')$.
\end{Rem}

\section{Perfect infinite fields with characteristic $2$}\label{section:car2}

In this final section, we tackle the case of infinite perfect fields with characteristic $2$.
Our aim is to prove Theorem \ref{theo:car2}, which we recall below:

\begin{theo}
Let $\F$ be a perfect infinite field of characteristic $2$. Assume that $\F$ is not quadratically closed.
Then $t_n(\F) = \dbinom{n+1}{2}$ for all $n \geq 1$.
Morever, the only irreducible optimal weakly triangularizable matrix spaces are $\Mat_1(\F)$ and $\mathfrak{sl}_2(\F)$.
\end{theo}

We assume from now on that $\F$ is a perfect infinite field of characteristic $2$, but not a quadratically closed one.
The proof of Theorem \ref{theo:car2} is by induction on $n$.

\subsection{The $n=2$ case}\label{section:n=2car2}

We start with the case $n=2$, which is obvious for the dimension part: not all $2$-by-$2$ matrices are triangularizable since $\F$
is not quadratically closed, yet every matrix of $\mathfrak{sl}_2(\F)$ has its characteristic polynomial of the form $t^2-\alpha$ for some $\alpha \in \F$, and hence is
triangularizable because $\F$ is perfect of characteristic $2$.

Conversely, let $\calH$ be a  weakly triangularizable $3$-dimensional linear subspace of $\Mat_2(\F)$.
Then $I_2 \in \calH$, because otherwise $\F I_2\oplus \calH$ would be weakly triangularizable and with dimension $4$.
Hence $\calH^\bot=\F B$ for some $B \in \Mat_2(\F) \setminus \{0\}$ with trace $0$.
If $B\in \F I_2$ then $\calH=\mathfrak{sl}_2(\F)$ and we are done. Otherwise $B$ is cyclic, and by conjugating $\calH$ if necessary we can assume that $B=\begin{bmatrix}
0 & 1 \\
\beta & 0
\end{bmatrix}$ for some $\beta \in \F$.
Hence, for all $(x,y,z) \in \F^3$, the matrix
 $\begin{bmatrix}
x & y \\
\beta y & z
\end{bmatrix}$ belongs to $\calH$, and hence is triangularizable over $\F$.

Assume that $\beta \neq 0$. Then we take $x=0$, $y=1$ and an arbitrary $z$,
and we see that all the polynomials $t^2-zt+\beta$ split over $\F$,
and hence by scaling $t^2-t+\beta/z^2$ would have a root in $\F$ for all $z \in \F^\times$.
Since $\F$ is perfect this would yield that $t^2-t+\gamma$ has a root in $\F$ for all $\gamma \in \F^\times$, and by scaling
we would deduce that $t^2-ut+\gamma$ has a root in $\F$ for all $u \in \F^\times$ and all $\gamma \in \F^\times$.
Since $\F$ is perfect this would yield that $\F$ is quadratically closed, thereby contradicting our starting assumption.

Hence $\beta=0$ and we obtain that $\calH=\MatT_2(\F)$, which is reducible.
We conclude that $\mathfrak{sl}_2(\F)$ is the sole irreducible weakly triangularizable $3$-dimensional linear subspace of $\Mat_2(\F)$.

\subsection{Start of the inductive step}

For the inductive step we will use a similar idea as in Section \ref{section:optimal}, but with an additional trick.
Here we will completely break the habit of trying to prove the inequality of dimensions first and
analyse the optimal spaces afterwards. We will take spaces with dimension equal or greater by one unit (!) than the dimension
we believe is the critical one, and then we will be able to control the structure of these spaces and discard the second possibility.

The idea is that, by assuming a dimension not too great, we can have a good understanding of the orthogonal complement of the space.
Let $n \geq 3$, let $V$ be an $n$-dimensional vector space, and let $\calS$ be a
weakly triangularizable subspace of $\End(V)$. We assume further that $\calS$ is irreducible and that
$$\dim \calS=\dbinom{n+1}{2}+i \quad \text{for some $i \in \{0,1\}$.}$$
We will prove that the case $i=1$ leads to a contradiction and that the case $i=0$ leads to \ldots yet another contradiction
(but it will take more time to reach it!).
We must immediately explain why this will allow us to climb the inductive step.

Indeed, say that we have a weakly triangularizable subspace $\calS$ of $\End(V)$ with dimension greater than $\dbinom{n+1}{2}$.
We can extract a linear subspace $\calT$ of $\calS$ with dimension $\dbinom{n+1}{2}+1$; of course $\calT$ is weakly triangularizable.
Then the claimed result yields that $\calT$ is reducible, and hence we obtain $\dim \calT \leq \dbinom{n+1}{2}$
by considering a reduced form and applying the induction hypothesis to the diagonal blocks.
Let us explain this in details: if we have a nontrivial $\calT$-invariant subspace $W$ with dimension $d \in \lcro 1,n-1\rcro$,
then for every $u \in \calT$ we consider the induced endomorphism $u_W$ and $u^{V/W}$ of $W$ and $V/W$, respectively.
Then $\{u_W \mid u \in \calS\}$ and $\{u^{V/W} \mid u \in \calS\}$ are weakly triangularizable subspaces of $\End(W)$ and $\End(V/W)$, respectively,
and the kernel of the mapping $u \in \calT \mapsto (u_W,u^{V/W})$ is included in the space of all linear mappings $u \in \End(V)$
that vanish everywhere on $W$ and map into $W$, a space which has dimension $d(n-d)$. Then
$\dim \calT \leq t_d(\F)+t_{n-d}(\F)+d(n-d)$, and hence by induction
$\dim \calT \leq \dbinom{d+1}{2}+\dbinom{n-d+1}{2}+d(n-d)=\dbinom{n+1}{2}$.
Hence we obtain $t_n(\F)=\dbinom{n+1}{2}$, and finally the above claim yields that there is no irreducible optimal
weakly triangularizable subspace of $\End(V)$.

Now, let us proceed according to the above plan. From now on, $\calS$ is an irreducible weakly triangularizable subspace
of $\End(V)$ with dimension $\dbinom{n+1}{2}+i$ for some $i \in \{0,1\}$.
First of all, we note that Claim \ref{claim:Sbotreduced} is still true in the present context, merely by the irreducibility of $\calS$
(as it relies only upon the Erasure Lemma, which holds for all fields that are not quadratically closed).

\begin{claim}\label{claim:vectorxcar2}
All the operators in $\widehat{\calS^\bot}$ are nonsurjective.
Moreover, if $i=1$ then $\maxrk (\widehat{\calS^\bot}) \leq n-2$.
\end{claim}

\begin{proof}
Let $x \in V \setminus \{0\}$. Remember the notation $\widehat{x} : u \in \calS^\bot \mapsto u(x) \in V$.
Just like in the proof of Claim \ref{claim:vectorx}, we find
a weakly triangularizable linear subspace $\calH$ of $\End(V/\F x)$ (namely
$\calH=\{u^{V/\F x} \mid u \in \calS \; \text{such that}\; u(x) \in \F x\}$) such that
$$\dim \calS \leq n-1+\dim(\calS \cap \Hom(V,\F x))+\dim \calH = n-1+(n-\rk \widehat{x})+\dim \calH.$$
By induction $\dim \calH \leq \dbinom{n}{2}$ and hence
$$n-\rk \widehat{x} \geq i+\dbinom{n+1}{2}-\dbinom{n}{2}-n+1 \geq i+1.$$
This yields the claimed results.
\end{proof}

\subsection{Applying Atkinson's theorem (part 1)}

Now, we can try to apply Atkinson's theorem to the dual operator space
$$\calT:=\widehat{\calS^\bot}.$$

Let us start with the case where $\calT$ is target-semiprimitive.
We have $\dim(\calS^\bot) \geq \dbinom{n}{2}-1$,
and we note that $\dbinom{n}{2}-1 > \dbinom{n-1}{2}+1$ if (and only if!) $n>3$.
Now, if $n=3$ and $\dim \calS^\bot = \dbinom{n}{2}-1$, then Claim \ref{claim:vectorxcar2} shows that
$\maxrk(\calT) \leq 1$. Yet it is folklore\footnote{In \cite{Atkinson} this is the statement ``There are no indecomposable $1$-spaces" at the bottom of page 314.} that for a (nonzero) vector space of linear operators with rank at most $1$, the nonzero operators have the same range or the same kernel
(an easy consequence of the fact that the sum of any two rank $1$ operators with distinct kernels and distinct ranges has rank $2$).
This implies that $\calT$ is not reduced because $\dim \calS^\bot>1$ and $\dim V>1$.
However $\calT$ is source-reduced (as is any dual operator space), whereas it is target-reduced because so is $\calS^\bot$
(we have explained that $\calS^\bot$ is reduced right before stating Claim \ref{claim:vectorxcar2}).

Hence we always have $\dim \calS^\bot > \dbinom{n-1}{2}+1$,
and in particular the last statement in Theorem \ref{theo:Atkinson} shows that $\maxrk (\calT)=n-1$, and hence
$i=0$ by Claim \ref{claim:vectorxcar2}. Hence we actually have $\dim \calS^\bot=\dbinom{n}{2}$, and from there
the remainder of the proof is strictly similar to the one from Section \ref{section:optimalAtkinson},
up to the point where $\calS$ is represented in some basis by the matrix space $\Mats_n(\F)P$ for some invertible symmetric matrix $P \in \GL_n(\F)$.

If $P$ is alternating, we proceed as in Section \ref{section:optimalAtkinson} to find a contradiction, noting here that $n>2$
(in that case we only used the fact that $\F$ is not quadratically closed).
Hence $P$ is congruent to a diagonal matrix, and as in Section \ref{section:optimalAtkinson} we deduce that
$\calS$ is represented by $\Mats_n(\F)$. Yet this is actually impossible.
Indeed, if true it would follow (by considering matrices $S \oplus 0_{n-2}$ with $S \in \Mats_2(\F)$) that
every matrix of $\Mats_2(\F)$ is triangularizable over $\F$, and we have seen in Section \ref{section:n=2car2}
that this fails because $\Mats_2(\F)$ is irreducible and $\Mats_2(\F) \neq \mathfrak{sl}_2(\F)$.

This final contradiction completes the proof in the case where $\calT$ is target-semiprimitive.

\subsection{Applying Atkinson's theorem (part 2)}

We complete the proof of Theorem \ref{theo:car2} by
examining the case where $\calT$ is not target-semiprimitive.
Then we follow the same routine as in the study of Case 2 in Section \ref{section:optimalAtkinson}, and we keep the same data
$(W,\pi,\calH,\calH')$ throughout, with $s:=\dim(V/W)$.
Once more, we can apply the first part of Atkinson's theorem to obtain
$$\dim \calH' \leq \dbinom{s}{2}.$$
Besides, by defining
$$\calS':=\{u \in \calS : u_{|W}=0\}$$
and by induction we find
$$\dim \calS' \leq s(n-s)+\dbinom{s+1}{2}=sn-\dbinom{s}{2}.$$
Besides, combining the rank theorem with Lemma \ref{lemma:orthocomplement} applied to $\calS^\bot$, we find
\begin{align*}
\dim \calS' & =\dim \calS- \dim \{u_{|W} \mid u \in \calS\} \\
 & = \dim \calS-n(n-s)+\dim \calH \\
 & = n^2-\dim \calS^\bot-n(n-s)+\dim \calH \\
 & = sn-\dim \calH'.
\end{align*}
Hence $\dim \calH' \geq \dbinom{s}{2}$ and we deduce that $\dim \calH'=\dbinom{s}{2}$ and that $\dim \calS'=s(n-s)+\dbinom{s+1}{2}$.

Now we move directly to the point where we have found a basis $\bfB$ of $V$ whose first $n-s$ vectors span $W$, and such that matrix space $\calM$ that is associated with $\calS$
satisfies the following properties:
\begin{enumerate}[(i)]
\item The space $\calM$ contains every matrix of the form $\begin{bmatrix}
[0]_{(n-s) \times (n-s)} & C \\
[0]_{s \times (n-s)} & [0]_{s \times s}
\end{bmatrix}$ with $C \in \Mat_{n-s,s}(\F)$.
\item There is an optimal weakly triangularizable subspace $\calD$ of $\Mat_s(\F)$ such that $\calM$ contains every matrix of the form
$\begin{bmatrix}
[0]_{(n-s) \times (n-s)} & [0]_{(n-s) \times s} \\
[0]_{s \times (n-s)} & D
\end{bmatrix}$ with $D \in \calD$.
\item Every matrix of $\calM$ has its lower-right block in $\calD$.
\end{enumerate}
It remains to prove that $W$ is invariant under all the elements of $\calS$, i.e.\ that all the matrices of $\calM$ have their lower-left block equal to zero.
This is where the proof differs from the one of Section \ref{section:optimalAtkinson}, as the Erasure Lemma cannot be used directly.
Indeed, here the induction hypothesis now yields that $\calD$ is conjugated to a join
of spaces of the form $\Mat_1(\F)$ and $\mathfrak{sl}_2(\F)$: Without loss of generality we can
modify the last $s$ vectors of $\bfB$ so that $\calD$ actually equals such a join. But now it is no longer true that
$\calD$ contains all the unit matrices $E_{i,i}$ with $i \in \lcro 1,s\rcro$ (because of the $\mathfrak{sl}_2(\F)$ cells).
Yet, by using the same strategy as in Section \ref{section:optimalAtkinson},
it is clear that it suffices to prove the following variation of the Erasure Lemma:

\begin{lemma}[Special Erasure Lemma]
Let $k \geq 3$ be an integer.
Let $C \in \Mat_{k-2,2}(\F)$ and $D \in \Mat_{k-2}(\F)$.
Assume that for every $A \in \mathfrak{sl}_2(\F)$ and every $B \in \Mat_{2,k-2}(\F)$, the matrix
$$\begin{bmatrix}
A & B \\
C & D
\end{bmatrix}$$
is triangularizable. Then $C=0$.
\end{lemma}

\begin{proof}
The result is actually deduced from the Erasure Lemma. Denote by $C_1$ the first column of $C$.
Take $A=\begin{bmatrix}
\alpha & 0 \\
0 & \alpha
\end{bmatrix}$
and an arbitrary matrix $B$ with second row equal to zero and an arbitrary first row $R$. Denoting by $(e_1,\dots,e_k)$
the standard basis of $\F^k$, this has the effect of making the subspace $\Vect(e_1,e_3,\dots,e_k)$ invariant under $M:=\begin{bmatrix}
A & B \\
C & D
\end{bmatrix}$, with resulting matrix equal to
$\begin{bmatrix}
\alpha & R \\
C_1 & D
\end{bmatrix}$. Hence the latter is triangularizable whatever the choices of $\alpha$ and $R$.
Since $\F$ is not quadratically closed, the Erasure Lemma yields $C_1=0$. Likewise the second column of $C$ vanishes.
\end{proof}

Applying the Special Erasure Lemma finally shows that $\calS$ is reducible, which contradicts the starting assumption. Hence the inductive step
is climbed: there is no irreducible weakly triangularizable subspace of $\End(V)$ of dimension either $\dbinom{n+1}{2}$ or $\dbinom{n+1}{2}+1$.
As we have explained in the beginning of Section \ref{section:car2}, this is enough to recover the equality $t_n(\F) = \dbinom{n+1}{2}$ and
the lack of an irreducible optimal weakly triangularizable subspace of $\Mat_n(\F)$.
Hence Theorem \ref{theo:car2} is proved, which in particular completes the proof of Theorem \ref{theo:dim}.

\appendix

\section{Appendix. A theory of optimal spaces}

Here, the setting is broader than in remainder of the present article. We consider
a division ring $\D$, i.e.\ a nontrivial unital associative ring in which every nonzero element is invertible,
and we assume that $\D$ is finite-dimensional over its center $C$. We take a subfield $\F$ of $C$ over which $\D$ has finite degree $d$.
We consider finite-dimensional \emph{right} vector space over $\D$, and matrices with entries in $\D$.
The additive group $\Mat_{n,p}(\D)$ of all $n$-by-$p$ matrices with entries in $\D$ is naturally seen as an $\F$-vector space (of dimension $dnp$).

For a right vector space $V$ over $\D$, a subset $\calX$ of $\End_\D(V)$ is called \textbf{irreducible} when $V$ is nonzero and $\calX$ has no proper invariant $\D$-linear subspace,
and $\calX$ is called reducible otherwise.

In this setting, we consider a \textbf{general property on endomorphisms}, that is a property $\calP(V,u)$ in two variables $V$ and $u$, where $V$ ranges over the finite-dimensional $\D$-vector spaces
and $u$ ranges over $\End_\D(V)$, which is invariant under similarity, meaning that for every isomorphism $\varphi : V \overset{\simeq}{\rightarrow} V'$ of finite-dimensional
$\D$-vector spaces and every $u \in \End_\D(V)$, the conditions $\calP(V,u)$ and $\calP(V',\varphi u \varphi^{-1})$ are equivalent.
We will simply say that an endomorphism $u$ of a vector space $V$ has property $\calP$ when $\calP(V,u)$ holds true,
and we will say that a square matrix $A \in \Mat_n(\D)$ has property $\calP$ when the endomorphism $X \in \D^n \mapsto AX \in \D^n$
has property $\calP$, i.e.\ any endomorphism of a right vector space (of dimension $n$) over $\D$ that is represented by $A$
has property $\calP$.

Now, we make the following critical assumption:

\begin{Def}
A general property on endomorphisms $\calP(V,u)$ is said to satisfy the \textbf{strong inheritance condition} whenever, for every finite-dimensional $\D$-vector space
$V$, every $u \in \End_\D(V)$ and every nontrivial $\D$-linear subspace $W$ of $V$ that is invariant under $u$, we have
$\calP(V,u)$ if and only if $\calP(W,u_W)$ and $\calP(V/W,u^{V/W})$, where $u_W$ stands for the endomorphism of $W$ induced by $u$, and
$u^{V/W}$ stands for the endomorphism of $V/W$ induced by $u$.
\end{Def}

Typical examples of general properties on endomorphisms that satisfy the strong inheritance condition are:
\begin{itemize}
\item $\F$-triangularizability (the property ``$u$ is annihilated by some split polynomial of $\F[t]$");
\item nilpotence (the property $\exists k \geq 1 : \; u^k=0$);
\item the $\F$-trivial spectrum property, i.e.\ $u$ has no non-zero eigenvalue in $\F$.
\end{itemize}
For the last property, note that since $\D$ has finite dimension over $\F$, one can see $u$ as an endomorphism of a finite-dimensional
vector space over $\F$.

Now, let us take such a property $\calP(V,u)$. Given a finite-dimensional $\D$-vector space $V$,
an $\F$-linear subspace of $\End_\D(V)$ is said to have property $\calP$ when all its elements have property $\calP$,
and likewise for $\F$-linear subspaces of $\Mat_n(\D)$.
Such a linear subspace is called \textbf{optimal} (with respect to $\calP$) when it has the greatest possible dimension
(over $\F$) among the $\F$-linear subspaces that have property $\calP$.

Now, let $\calS$ be an $\F$-linear subspace of $\End_\D(V)$ with property $\calP$.
Consider the simple situation where we have a nontrivial $\calS$-invariant subspace $W$.
It is a bit easier to think in terms of matrices, so say for a moment that $V=\D^n$ for some $n \geq 2$, and $W=\D^p \times \{0\}$ for some
$p \in \lcro 1,n-1\rcro$. The elements of $\End_\D(V)$ are identified with matrices in the usual way (the endomorphism that corresponds to the matrix $A$
is $X \mapsto AX$). So, the assumption that $W$ is $\calS$-invariant means that every
$M \in \calS$ reads
$$M=\begin{bmatrix}
A(M) & B(M) \\
[0] & C(M)
\end{bmatrix} \quad \text{where $A(M) \in \Mat_p(\D)$ and so on.}$$
We can naturally identify $A(\calS)$ with an $\F$-linear subspace of $\End_\D(\D^p)$, and $C(\calS)$ with an $\F$-linear subspace of $\End_\D(\D^{n-p})$.
By the strong inheritance condition (and because of the invariance under similarity)
the space $\calS$ has property $\calP$ if and only if both $A(\calS)$ and $C(\calS)$ have property $\calP$.
Moreover, $\calS$ is included in the \emph{joint} of $A(\calS)$ and $C(\calS)$, defined as follows:

\begin{Def}
Let $\calA$ and $\calC$ be respective subsets of $\Mat_n(\D)$ and $\Mat_p(\D)$. The \textbf{joint} $\calA \vee \calC$ is defined as
the set of all matrices in $\Mat_{n+p}(\D)$ of the form
$$\begin{bmatrix}
A & B \\
[0]_{p \times n} & C
\end{bmatrix} \quad \text{with $A \in \calA$, $B \in \Mat_{n,p}(\D)$ and $C \in \calC$.}$$
If $\calA$ and $\calC$ are $\F$-linear subspaces, their joint is an $\F$-linear subspace and
$$\dim_\F(\calA \vee \calC)=npd+\dim_\F \calA+\dim_\F \calC.$$
\end{Def}

Thus $\calS \subseteq A(\calS) \vee C(\calS)$. Moreover the strong inheritance condition yields that
$A(\calS) \vee C(\calS)$ has property $\calP$. Since $\calS$ is optimal, we deduce that $\calS = A(\calS) \vee C(\calS)$
and it follows in particular that $\calS$ contains all the operators $u \in \End_\D(V)$ that vanish everywhere on $W$ and map into $W$.
Finally, if $A(\calS)$ were not optimal, then we could choose an optimal subspace $\calT$ of $\End_\D(W)$ with property
$\calP$ and $\dim_\F \calT>\dim_\F A(\calS)$, and by using the strong inheritance condition once more, we would retrieve that the space $\calT \vee C(\calS)$
has property $\calP$, contradicting the optimality of $\calS$ since $\dim_\F (\calT \vee C(\calS))>\dim_\F \calS$.
Hence $A(\calS)$ is optimal, and likewise $C(\calS)$ is optimal.

Now, let us come back to the geometric version and let us generalize the previous idea.
First of all, a \textbf{flag} of $\D$-linear subspaces of $V$ is an increasing list $(V_0,\dots,V_p)$ of $\D$-linear subspaces such that
$V_0=\{0\}$ and $V_p=V$ (we call such a flag complete if $\dim_\D(V_{i+1}/V_i)=1$ for all $i \in \lcro 0,p-1\rcro$).
Let $\calF=(V_0,\dots,V_p)$ be a (potentially incomplete) flag of $\D$-linear subspaces of $V$.
Let $\calS$ be an $\F$-linear subspace of $\End_\D(V)$, and assume that all the subspaces $V_i$ are $\calS$-invariant.
Then every $u \in \calS$ induces an endomorphism $u_i$ of $V_i/V_{i-1}$ for all $i \in \lcro 1,p\rcro$.
This way, we obtain respective $\F$-linear subspaces $\calS_{\calF,1},\dots,\calS_{\calF,p}$
of $\End_\D(V_1/V_0),\dots,\End_\D(V_p/V_{p-1})$, and we find the inclusion
$\calS \subseteq \calS_{\calF,1}\vee \cdots \vee\calS_{\calF,p}$, where the joint of
respective subsets $\calT_1 \subseteq \End_\D(V_1/V_0),\dots,\calT_p \subseteq \End_\D(V_p/V_{p-1})$, denoted by
$\calT_1\vee \cdots \vee \calT_p$, is defined as the set of all $u \in \End_\D(V)$ that leave $V_1,\dots,V_p$ invariant
and induce respective elements of $\calT_1,\dots,\calT_p$ on $V_1/V_0,\dots,V_p/V_{p-1}$.

Moreover, given $u \in \calT_1\vee \cdots \vee \calT_p$, we gather from the strong inheritance condition that $u$ has property $\calP$
 if and only if $u_1,\dots,u_p$ \emph{all} have property $\calP$. Hence
$\calT_1\vee \cdots \vee \calT_p$ has property $\calP$ if and only if $\calT_1,\dots,\calT_p$ all have property $\calP$.
Like in the above matrix case, we prove that
if $\calS$ is an optimal subspace with property $\calP$ then so are $\calS_{\calF,1},\dots,\calS_{\calF,p}$, and
$$\calS=\calS_{\calF,1}\vee \dots\vee \calS_{\calF,p}.$$
Beware that the converse might not hold, i.e.\ it is possible in theory that $\calS_{\calF,1},\dots,\calS_{\calF,p}$
be optimal spaces with property $\calP$ but $\calS$ is not optimal.

Finally, for all the $\calS_{\calF,i}$ spaces to be irreducible, it is clearly necessary and sufficient that $\calF$
be a \emph{maximal} flag of $\calS$-invariant $\D$-linear subspaces.

We now prove that there is a unique such flag if $\calS$ is optimal. This is based on the following lemma:

\begin{lemma}\label{lemma:totallyordered}
Let $\calS$ be an optimal linear subspace of $\End_\D(V)$ with property $\calP$.
Then the $\calS$-invariant $\D$-linear subspaces of $V$ are totally ordered by inclusion.
\end{lemma}

\begin{proof}
Let $W$ be a nontrivial $\calS$-invariant $\D$-linear subspace.
Since $\calS$ is optimal, we gather from the above considerations that
$\calS =\calT \vee \calT'$ for some non-empty subset $\calT$ of $\End_\D(W/\{0\})$ and some non-empty subset
$\calT'$ of $\End_\D(V/W)$, and in particular $\calS$ contains all the elements of $\End_\D(V)$ that vanish everywhere on $W$ and map into $W$.
Consequently, for all $X \in V \setminus W$, we have $W \subseteq \calS X$.

Let $W'$ be a nontrivial $\calS$-invariant subspace. Assume that $W'$ is not included in $W$.
Then we choose $X \in W' \setminus W$, and since $W'$ is invariant under $\calS$ we find $W \subseteq \calS X \subseteq W'$.
\end{proof}

Now, we can conclude. Let $\calS$ be an optimal linear subspace of $\End_\D(V)$ with property $\calP$.
By Lemma \ref{lemma:totallyordered}, the $\calS$-invariant $\D$-linear subspaces of $V$
form a flag $\calF=(V_0,\dots,V_p)$ of $V$, and this flag is of course maximal among the flags of $\calS$-invariant $\D$-linear subspaces.
Combining this with the previous general study, we can conclude:

\begin{theo}
Let $\D$ be a division ring, and $\F$ be a subfield of the center of $\D$, such that $\D$ has finite degree over $\F$.
Let $\calP(V,u)$ be a property on endomorphisms of $\D$-vector spaces that satisfies the strong inheritance condition.

Let $V$ be a non-zero finite-dimensional vector space over $\D$, and $\calS$ be an optimal $\F$-linear subspace of $\End_\D(V)$
with property $\calP$. Then there is a unique flag $\calF=(V_0,V_1,\dots,V_p)$ of $\calS$-invariant subspaces and, for all
$i \in \lcro 1,p\rcro$, a unique irreducible optimal $\F$-linear subspace $\calT_i$ of $\End_\D(V_i/V_{i-1})$ with property $\calP$ such that
$\calS=\calT_1 \vee \cdots \vee \calT_p$.
\end{theo}

For matrices, the previous theorem is translated as follows:

\begin{theo}\label{theo:generalnonsense}
Let $\D$ be a division ring, and $\F$ be a subfield of the center of $\D$, such that $\D$ has finite degree over $\F$.
Let $\calP(V,u)$ be a property on endomorphisms that satisfies the strong inheritance condition.
Let $n>0$ be a positive integer, and $\calM$ be an optimal $\F$-linear subspace of $\Mat_n(\D)$ with property $\calP$.

Then there is a matrix $P \in \GL_n(\F)$, a partition $n=n_1+\cdots+n_p$ into positive integers, and for each $i \in \lcro 1,p\rcro$
an irreducible optimal $\F$-linear subspace $\calM_i$ of $\Mat_{n_i}(\D)$ such that
$$P \calM P^{-1}=\calM_1 \vee \calM_2 \vee \cdots \vee \calM_p.$$
The list $(n_1,\dots,n_p)$ is then uniquely determined by the similarity class of $\calM$, and so is every space $\calM_i$ up to similarity.
\end{theo}

The bottom line of the above theorem is that the study of optimal spaces with property $\calP$
is entirely reduced to the one of optimal irreducible spaces with property $\calP$.
But there is (small) caveat, which we call the \textbf{composition problem}, and it is probably better understood through the matrix viewpoint. If we have a partition
$n=n_1+\cdots+n_p$ into positive integers,
and for each $i \in \lcro 1,p\rcro$ we have an irreducible optimal $\F$-linear subspace $\calM_i$ of $\Mat_{n_i}(\D)$ with property $\calP$, there is no general way to obtain that the space
$\calM_1 \vee \calM_2 \vee \cdots \vee \calM_p$ is optimal: This appears to depend on the choice of the property $\calP$.
In most situations, it turns out that the composition problem has a positive answer. Here are known properties where the answer is positive:
\begin{itemize}
\item $\calP$ is the nilpotence property (with an arbitrary division ring, see \cite{dSPGerstenhaberskew});
\item $\calP$ is the $\F$-trivial spectrum property with $\D=\F$ (see \cite{dSPlargerank});
\item $\calP$ is the $\F$-triangularizability property with $\D=\F$, where $\F$ is either an infinite perfect field of characteristic $2$ that is not quadratically closed,
or a field with characteristic other than $2$ that is not quadratically closed.
\end{itemize}
For each one of these properties indeed, if we denote by $\delta_n$ the dimension (over $\F$) of the optimal $\F$-linear subspaces of $\Mat_n(\D)$ with property $\calP$,
the identity $\delta_{n+p}=\delta_n+\delta_p+dnp$ is known to hold for all integers $n \geq 1$ and $p \geq 1$.
For the third example, which is dealt with in the present article, this comes from the identity $t_n(\F)=\dbinom{n+1}{2}$ proved in Theorem \ref{theo:dim} for such cases.

\newpage

\section{On the triangularizability of symmetric matrices over fields with characteristic $2$}\label{section:appendixB}

Here, we prove the following result:

\begin{theo}
Let $\F$ be an NRC field of characteristic $2$.
Then:
\begin{enumerate}[(a)]
\item $\Mats_2(\F)$ is weakly triangularizable if and only if every separable polynomial of degree $2$ over $\F$ splits.
\item $\Mats_3(\F)$ is not weakly triangularizable, and consequently $\Mats_n(\F)$ is not weakly triangularizable whatever the integer $n>2$.
\end{enumerate}
\end{theo}

\begin{proof}
Note that every trace zero matrix of $\Mats_2(\F)$ has its characteristic polynomial of the form $t^2-a^2+b^2=(t-a+b)^2$ for some $(a,b)\in \F^2$, and hence
is triangularizable over $\F$. Hence, if $\Mats_2(\F)$ is not weakly triangularizable then it contains a matrix $M \in \Mats_2(\F)$
such that $\tr(M)\neq 0$ and $\chi_M$ is irreducible, and $\chi_M$ is separable with degree $2$.

Conversely, assume that $\Mats_2(\F)$ is weakly triangularizable. To start with, we prove that $x \in \F \mapsto x^2+x \in \F$
is surjective. Let $y \in \F$. Since $\begin{bmatrix}
1 & y \\
y & 0
\end{bmatrix}$ is symmetric, it is triangularizable; its characteristic polynomial equals $t^2+t+y^2$, and hence there exists $x \in \F$ such that $x^2+x=y^2$.
Then $(y+x)^2+(y+x)=y^2+x^2+x+y=y$.
Hence the claimed surjectivity. Finally, let $\alpha \in \F \setminus \{0\}$ and $y \in \F$.
Then we can find $x \in \F$ such that $x^2+x=y/\alpha^2$, to the effect that $(\alpha x)^2+\alpha (\alpha x)=y$.
Hence, every polynomial with degree $2$ and nonzero trace over $\F$ has a root in $\F$. This completes the proof of point (a).

To prove point (b), we use the fact that $\F$ is NRC, taking a scalar $\lambda \in \F$ for which $t^2-\lambda$ has no root in $\F$.
We consider the matrix $A:=\begin{bmatrix}
0 & 0 & 0 \\
0 & 0 & \lambda \\
0 & 1 & 0
\end{bmatrix}$, which has characteristic polynomial $t(t^2-\lambda)$ and is therefore non-triangularizable.
We shall prove that $A$ is similar to a symmetric matrix. To see this, we observe that
$A$ represents a selfadjoint map for the inner product $b : (X,Y) \mapsto X^T S Y$ on $\F^3$, where
$S:=\begin{bmatrix}
1 & 0 & 0 \\
0 & 0 & 1 \\
0 & 1 & 0
\end{bmatrix}$, which amounts to observing that $SA=\Diag(0,1,\lambda)$ is symmetric.
To conclude, it suffices to observe that the symmetric bilinear form $b$ is represented in some basis by $I_3$.
Since $S$ is not alternating yet invertible, $b$ has an orthogonal basis $(e_1,e_2,e_3)$ made of nonisotropic vectors.
Besides, as the diagonal entries of $S$ are squares, the values of the associated quadratic form $Q : X \mapsto X^T S X$ are squares in $\F$. Hence by scaling we can find a $b$-orthogonal basis $(e_1,e_2,e_3)$ of $\F^3$ in which $Q(e_1)=Q(e_2)=Q(e_3)=1$,
which yields the claimed statement. Then, the matrix representation of $X \mapsto AX$ in that basis is symmetric, and we have found a non-triangularizable matrix of 
$\Mats_3(\F)$. We leave the computation of an explicit such matrix as an exercise to the reader.
\end{proof}

An example of an NRC field with characteristic $2$ in which every $2$-by-$2$ symmetric matrix is triangularizable is the
separable closure $\K$ of $\F_2(t)$ (the field of fractions with one indeterminate over $\F_2$) in an algebraic closure $\overline{\F_2(t)}$.
Then every separable polynomial with coefficients in $\K$ splits over $\K$, but by the general theory of separable extensions 
it is known that $t$ has no more square root in $\K$ than it has in $\F_2(t)$ itself, hence $\K$ is NRC.

\section*{Conflict of interest statement}

The author states that there is no conflict of interest.

\section*{Data availability statement}

There is no relevant data corresponding to this manuscript.

\end{document}